\documentclass[12pt]{article}
\usepackage{amsmath,amsfonts,amssymb,amsthm}
\usepackage{enumerate}

\numberwithin{equation}{section}
\theoremstyle{plain}
\newtheorem{theorem}[equation]{Theorem}
\newtheorem{corollary}[equation]{Corollary}
\newtheorem{lemma}[equation]{Lemma}
\newtheorem{proposition}[equation]{Proposition}

\theoremstyle{definition}
\newtheorem{definition}[equation]{Definition}

\newtheorem{example}[equation]{Example}
\newtheorem{remark}[equation]{Remark}

\newtheorem*{openproblem}{Open Problem}

\numberwithin{equation}{section}

\newcommand{\R}{{\mathbb R}}
\newcommand{\N}{{\mathbb N}}

\newcommand{\Om}{\Omega}

\providecommand{\vint}[1]{\mathchoice
          {\mathop{\vrule width 5pt height 3 pt depth -2.5pt
                  \kern -9pt \kern 1pt\intop}\nolimits_{\kern -5pt{#1}}}
          {\mathop{\vrule width 5pt height 3 pt depth -2.6pt
                  \kern -6pt \intop}\nolimits_{\kern -3pt{#1}}}
          {\mathop{\vrule width 5pt height 3 pt depth -2.6pt
                  \kern -6pt \intop}\nolimits_{\kern -3pt{#1}}}
          {\mathop{\vrule width 5pt height 3 pt depth -2.6pt
                  \kern -6pt \intop}\nolimits_{\kern -3pt{#1}}}}

\newcommand{\eps}{\varepsilon}
\newcommand{\loc}{\mathrm{loc}}

\newcommand{\BV}{\mathrm{BV}}
\newcommand{\liploc}{\mathrm{Lip}_{\mathrm{loc}}}

\newcommand{\ch}{\text{\raise 1.3pt \hbox{$\chi$}\kern-0.2pt}}

\DeclareMathOperator{\capa}{Cap}
\DeclareMathOperator{\rcapa}{cap}

\DeclareMathOperator{\dist}{dist}

\DeclareMathOperator{\Lip}{Lip}

\DeclareMathOperator{\supp}{spt}

\DeclareMathOperator{\inte}{int}

\begin{document}
\title{The variational $1$-capacity and\\
BV functions with zero boundary values\\
on metric spaces
\footnote{{\bf 2010 Mathematics Subject Classification}: 30L99, 31E05, 26B30.
\hfill \break {\it Keywords\,}: metric measure space, bounded variation, zero boundary values,
variational capacity, outer capacity, quasi-semicontinuity
}}
\author{Panu Lahti}
\maketitle

\begin{abstract}
In the setting of a metric space that is equipped with a doubling measure and supports a Poincar\'e inequality, we define and study a class of $\BV$ functions with zero boundary values.
In particular, we show that the class is the closure of compactly supported $\BV$
functions in the $\BV$ norm.
Utilizing this theory, we then study the variational $1$-capacity
and its Lipschitz and $\BV$ analogs. We show that each of these is an outer capacity, and that
the different capacities are equal for certain sets.
\end{abstract}

\section{Introduction}

Spaces of Sobolev functions with zero boundary values are essential in specifying boundary values
in various Dirichlet problems.
This is true also in the setting of a metric measure
space $(X,d,\mu)$, where $\mu$ is a doubling Radon measure
and the space supports a
Poincar\'e inequality; see Section \ref{sec:definitions} for definitions and notation.
In this setting, given an open
set $\Om\subset X$, the space of \emph{Newton-Sobolev functions}
with zero boundary values is defined
for $1\le p<\infty$ by
\[
N_0^{1,p}(\Om):=\{u|_{\Om}:\,u\in N^{1,p}(X)\textrm{ with }u=0\textrm { on }X\setminus\Om\}.
\]
Dirichlet problems for minimizers of the $p$-energy, and Newton-Sobolev functions with zero boundary values have been studied in the metric setting in \cite{BB-OD,BBS2,BBS3,S-H}.

In the case $p=1$, instead of the $p$-energy it is natural to minimize the total variation of a
function.
Local minimizers of the total variation are called functions of least gradient,
see e.g. \cite{BDG,HKLS,MRL,SWZ,ZZ}.
To study these, or alternatively solutions to Dirichlet problems that minimize the total
variation globally, we need a class of functions of bounded variation ($\BV$ functions)
with zero boundary values.
Such a notion has been considered in the Euclidean setting in e.g.
\cite{BS} and in the metric setting in \cite{HKL,KLLS,LaSh2}.
However, unlike in the case $p>1$, for $\BV$ functions there seem to be several natural ways to define the notion of zero boundary values, depending for example on whether one considers
local or global minimizers.
In this paper we define the class $\BV_0(\Om)$ in a way that mimics the definition of the classes
$N_0^{1,p}(\Om)$ as closely as possible; we expect such a definition to be useful when
extending results of fine potential theory from the case $p>1$ to the case $p=1$,
see Remark \ref{rem:definitions of zero boundary values}.
Then we show that various properties that are known to hold for $N_0^{1,p}(\Om)$ hold also for
$\BV_0(\Om)$.

Classically, the space of Sobolev functions with zero boundary values is usually defined as
the closure of $C_0^{\infty}(\Om)$ in the Sobolev norm.
In the metric setting, it can be shown that the space $N_0^{1,p}(\Om)$ is the closure of
the space of Lipschitz functions with compact support in $\Om$,
see \cite[Theorem 4.8]{S-H} or \cite[Theorem 5.46]{BB}. In this
paper we show that
the class $\BV_0(\Om)$ is, analogously, the closure of $\BV$ functions
with compact support in $\Om$.
This is Theorem \ref{thm:characterization of BV function with zero bdry values}.

Newton-Sobolev classes with zero boundary values are needed in defining the variational capacity
$\rcapa_p(A,\Om)$, which is an essential concept in nonlinear potential theory, see e.g.
the monographs
\cite{HKM,MZ} for the Euclidean case and \cite{BB} for the metric setting.
The properties of the variational capacity $\rcapa_p(A,D)$, also for
nonopen $D$, have been studied systematically in the metric setting in \cite{BB-cap}.
In this paper, we extend some of these results from the case $1<p<\infty$ to the case $p=1$.
In particular, in Theorem \ref{thm:1-capacity is outer} we show that the variational $1$-capacity
$\rcapa_1$ is an \emph{outer capacity}.

Moreover, the $\BV$ analog of the variational $1$-capacity, denoted by $\rcapa_{\BV}$, has been studied in the metric setting in \cite{HaSh,KKST-DG}.
Again, there are several different possible definitions available,
depending on the definition of the class of $\BV$ functions with zero boundary values,
and usually the $\BV$-capacity is defined
in a way that automatically makes it an outer capacity. In this paper we instead give a definition
that is closely analogous to the definition of $\rcapa_1$. Then we show, in Theorem
\ref{thm:BV capacity is outer}, that $\rcapa_{\BV}$ is in fact an outer capacity. Moreover, we show that
when $K$ is a compact subset of an open set $\Om$, $\rcapa_{\BV}(K,\Om)$ is equal to the Lipschitz
version of the $1$-capacity $\rcapa_{\mathrm{lip}}(K,\Om)$. This is
Theorem \ref{thm:BV and lip caps are equal}.

In the literature, proving that compactly supported
Lipschitz functions are dense in $N_0^{1,p}(\Om)$, as well as proving that $\rcapa_p$ is an outer capacity,
relies on the quasicontinuity of Newton-Sobolev functions, see \cite{BBS}.
In this paper, our main tool is a partially analogous
quasi-semicontinuity property of $\BV$ functions proved in \cite{LaSh} in the metric setting,
and previously in \cite[Theorem 2.5]{CDLP} in the Euclidean setting.

\section{Definitions and notation}\label{sec:definitions}

In this section we introduce the notation, definitions, and assumptions used in the paper.

Throughout this paper, $(X,d,\mu)$ is a complete metric space that is equipped
with a doubling Borel regular outer measure $\mu$ and satisfies a Poincar\'e inequality
defined below.
The doubling condition means that there is a constant $C_d\ge 1$ such that
\[
0<\mu(B(x,2r))\leq C_d\mu(B(x,r))<\infty
\]
for every ball $B=B(x,r)$ with center $x\in X$ and radius $r>0$.
By iterating the doubling condition, we obtain that for any $x\in X$ and $y\in B(x,R)$ with $0<r\le R<\infty$, we have
\begin{equation}\label{eq:homogenous dimension}
\frac{\mu(B(y,r))}{\mu(B(x,R))}\ge \frac{1}{C_d^2}\left(\frac{r}{R}\right)^{Q},
\end{equation}
where $Q>1$ only depends on the doubling constant $C_d$.
When we want to state that a constant $C$
depends on the parameters $a,b, \ldots,$ we write $C=C(a,b,\ldots)$.
When a property holds outside a set of $\mu$-measure zero, we say that it holds
almost everywhere, abbreviated a.e.

All functions defined on $X$ or its subsets will take values in $[-\infty,\infty]$.
Since $X$ is complete and equipped with a doubling measure, it is proper,
meaning that closed and bounded sets are compact.
Since $X$ is proper, given an open set $\Om\subset X$
we define $\liploc(\Omega)$ to be the space of
functions that are in the Lipschitz class $\Lip(\Omega')$ for every open set $\Omega'$ whose closure is a compact subset of $\Om$.
Other local spaces of functions are defined analogously.

The measure-theoretic boundary $\partial^{*}E$ of a set $E\subset X$ is the set of points $x\in X$
at which both $E$ and its complement have strictly positive upper density, i.e.
\begin{equation}\label{eq:measure theoretic boundary}
\limsup_{r\to 0}\frac{\mu(B(x,r)\cap E)}{\mu(B(x,r))}>0\quad
\textrm{and}\quad\limsup_{r\to 0}\frac{\mu(B(x,r)\setminus E)}{\mu(B(x,r))}>0.
\end{equation}

For any $A\subset X$ and $0<R<\infty$, the restricted spherical Hausdorff content
of codimension one is defined by
\[
\mathcal{H}_{R}(A):=\inf\left\{ \sum_{i=1}^{\infty}
  \frac{\mu(B(x_{i},r_{i}))}{r_{i}}:\,A\subset\bigcup_{i=1}^{\infty}B(x_{i},r_{i}),\,r_{i}\le R\right\}.
\]
The codimension one Hausdorff measure of $A\subset X$ is then defined by
\begin{equation*}
  \mathcal{H}(A):=\lim_{R\rightarrow 0}\mathcal{H}_{R}(A).
\end{equation*}

By a curve we mean a nonconstant rectifiable continuous mapping from a compact interval
into $X$.
We say that a nonnegative Borel function $g$ on $X$ is an upper gradient 
of a function $u$
on $X$ if for every curve $\gamma$, we have
\begin{equation}\label{eq:definition of upper gradient}
|u(x)-u(y)|\le \int_\gamma g\,ds,
\end{equation}
where $x$ and $y$ are the end points of $\gamma$,
and the curve integral is defined by using an arc-length parametrization,
see \cite[Section 2]{HK} where upper gradients were originally introduced.
We interpret $|u(x)-u(y)|=\infty$ whenever  
at least one of $|u(x)|$, $|u(y)|$ is infinite.

In the following, let $1\le p<\infty$ (later we will almost exclusively consider the case $p=1$).
We say that a family $\Gamma$ of curves is of zero $p$-modulus if there is a 
nonnegative Borel function $\rho\in L^p(X)$ such that 
for all curves $\gamma\in\Gamma$, the curve integral $\int_\gamma \rho\,ds$ is infinite.
A property is said to hold for $p$-almost every curve
if it fails only for a curve family with zero $p$-modulus. 
If $g$ is a nonnegative $\mu$-measurable function on $X$
and (\ref{eq:definition of upper gradient}) holds for $p$-almost every curve,
we say that $g$ is a $p$-weak upper gradient of $u$. 
By only considering curves $\gamma$ in a set $A\subset X$, we can talk about
a function $g$ being a ($p$-weak) upper gradient of $u$ in $A$.

Given a $\mu$-measurable set $D\subset X$, we define
\[
\Vert u\Vert_{N^{1,p}(D)}:=\left(\int_D |u|^p\,d\mu+\inf \int_D g^p\,d\mu \right)^{1/p},
\]
where the infimum is taken over all $p$-weak upper gradients $g$ of $u$ in $D$.
The substitute for the Sobolev space $W^{1,p}$ in the metric setting is the Newton-Sobolev space
\[
N^{1,p}(D):=\{u:\,\|u\|_{N^{1,p}(D)}<\infty\},
\]
which was introduced in \cite{S}.
We understand a Newton-Sobolev function to be defined at every $x\in D$
(even though $\Vert \cdot\Vert_{N^{1,p}(D)}$ is then only a seminorm).
For any $D\subset X$, the space of Newton-Sobolev functions with zero boundary values is defined as
\[
N_0^{1,p}(D):=\{u|_{D}:\,u\in N^{1,p}(X)\textrm{ and }u=0\textrm { on }X\setminus D\}.
\]
The space is a subspace of $N^{1,p}(D)$ when $D$ is $\mu$-measurable, and it can always
be understood to be a subspace of $N^{1,p}(X)$.

It is known that for any $u\in N^{1,p}_{\loc}(X)$, there exists a minimal $p$-weak
upper gradient of $u$, always denoted by $g_{u}$, satisfying $g_{u}\le g$ 
a.e. for any $p$-weak upper gradient $g\in L_{\loc}^{p}(X)$
of $u$, see \cite[Theorem 2.25]{BB}.

The $p$-capacity of a set $A\subset X$ is given by
\[
 \capa_p(A):=\inf \Vert u\Vert_{N^{1,p}(X)}^p,
\]
where the infimum is taken over all functions $u\in N^{1,p}(X)$ such that $u\ge 1$ on $A$.
By truncation we see that we can additionally require $0\le u\le 1$.
We know that $\capa_p$ is an outer capacity, meaning that
\[
\capa_p(A)=\inf\{\capa_p(U):\,U\supset A\textrm{ is open}\}
\]
for any $A\subset X$, see \cite[Theorem 5.31]{BB}.
If a property holds outside a set
$A\subset X$ with $\capa_p(A)=0$, we say that it holds $p$-quasieverywhere, or $p$-q.e.
If $u\in N^{1,p}(X)$, then
\begin{equation}\label{eq:quasieverywhere equivalence classes}
\Vert u-v\Vert_{N^{1,p}(X)}=0\quad \textrm{if and only if}\quad u=v\  p\textrm{-q.e.},
\end{equation}
see \cite[Proposition 1.61]{BB}. Thus in the definition of $N^{1,p}_0(D)$, we can
equivalently require that
$u=0$ $p$-q.e. on $X\setminus D$.
The variational $p$-capacity of a set $A\subset D$ with respect to a set $D\subset X$ is
\begin{equation}\label{eq:def of variational p-capacity}
\rcapa_p(A,D):=\inf \int_{X} g_u^p\,d\mu,
\end{equation}
where the infimum is taken over functions $u\in N_0^{1,p}(D)$ such that $u\ge 1$ on $A$ (equivalently,  $p$-q.e. on $A$).
For basic properties satisfied by the $p$-capacity and the variational $p$-capacity, such as
monotonicity and countable subadditivity, see \cite{BB,BB-cap}.

Next we recall the definition and basic properties of functions
of bounded variation on metric spaces, essentially following \cite{M}. See also e.g.
\cite{AFP, EvaG92, Fed, Giu84, Zie89} for the classical 
theory in the Euclidean setting.
Let $\Om\subset X$ be an open set.
For $u\in L^1_{\loc}(\Om)$, we define the total variation of $u$ in $\Om$ by
\[
\|Du\|(\Om):=\inf\left\{\liminf_{i\to\infty}\int_\Om g_{u_i}\,d\mu:\, u_i\in
\Lip_{\loc}(\Om),\, u_i\to u\textrm{ in } L^1_{\loc}(\Om)\right\},
\]
where each $g_{u_i}$ is the minimal $1$-weak upper gradient of $u_i$ in $\Om$.
Note that in \cite{M}, local Lipschitz constants were used instead of upper gradients, but
the properties of the total variation can be proved similarly with either definition.
We say that $u\in L^1(\Om)$ is a function of bounded variation,
and denote $u\in\BV(\Om)$, if $\|Du\|(\Om)<\infty$.
For an arbitrary set $A\subset X$, we define
\[
\|Du\|(A):=\inf\{\|Du\|(U):\, A\subset U,\,U\subset X
\text{ is open}\}.
\]
If $\Vert Du\Vert(\Omega)<\infty$, then
$\|Du\|(\cdot)$ is a finite Radon measure on $\Omega$ by
\cite[Theorem 3.4]{M}.
A $\mu$-measurable set $E\subset X$ is said to be of \emph{finite perimeter} if
$\|D\ch_E\|(X)<\infty$, where $\ch_E$ is the characteristic function of $E$.
The perimeter of $E$ in $\Omega$ is also denoted by
\[
P(E,\Omega):=\|D\ch_E\|(\Omega).
\]
The $\BV$ norm is defined by
\[
\Vert u\Vert_{\BV(\Om)}:=\Vert u\Vert_{L^1(\Om)}+\Vert Du\Vert(\Om).
\]
The $\BV$-capacity of a set $A\subset X$ is defined by
\[
\capa_{\BV}(A):=\inf \Vert u\Vert_{\BV(X)},
\]
where the infimum is taken over all $u\in\BV(X)$ with $u\ge 1$ in a neighborhood of $A$.

The following coarea formula is given in \cite[Proposition 4.2]{M}:
if $\Omega\subset X$ is an open set and $u\in L^1_{\loc}(\Omega)$, then
\begin{equation}\label{eq:coarea}
\|Du\|(\Omega)=\int_{-\infty}^{\infty}P(\{u>t\},\Omega)\,dt.
\end{equation}
If $\Vert Du\Vert(\Om)<\infty$, the above holds with $\Om$ replaced by any Borel set $A\subset \Om$.

We will assume throughout the paper  that $X$ supports a $(1,1)$-Poincar\'e inequality,
meaning that there exist constants $C_P\ge 1$ and $\lambda \ge 1$ such that for every
ball $B(x,r)$, every locally integrable function $u$ on $X$,
and every upper gradient $g$ of $u$,
we have 
\[
\vint{B(x,r)}|u-u_{B(x,r)}|\, d\mu 
\le C_P r\vint{B(x,\lambda r)}g\,d\mu,
\]
where 
\[
u_{B(x,r)}:=\vint{B(x,r)}u\,d\mu :=\frac 1{\mu(B(x,r))}\int_{B(x,r)}u\,d\mu.
\]
The $(1,1)$-Poincar\'e inequality implies the so-called Sobolev-Poincar\'e inequality, see e.g. \cite[Theorem 4.21]{BB}, and by applying the latter to approximating locally Lipschitz functions in the definition of the total variation, we get the following Sobolev-Poincar\'e inequality for $\BV$ functions. For every ball $B(x,r)$ and every $u\in L^1_{\loc}(X)$, we have
\begin{equation}\label{eq:sobolev poincare inequality}
\left(\,\vint{B(x,r)}|u-u_{B(x,r)}|^{Q/(Q-1)}\,d\mu\right)^{(Q-1)/Q}
\le C_{SP}r\frac{\Vert Du\Vert (B(x,2\lambda r))}{\mu(B(x,2\lambda r))},
\end{equation}
where $Q>1$ is the exponent from \eqref{eq:homogenous dimension} and
$C_{SP}=C_{SP}(C_d,C_P,\lambda)\ge 1$ is a constant.

Given an open set $\Omega\subset X$ and a $\mu$-measurable set $E\subset X$ with
$P(E,\Omega)<\infty$, for any $A\subset\Om$ we have
\begin{equation}\label{eq:def of theta}
\alpha \mathcal H(\partial^*E\cap A)\le P(E,A)\le C_d \mathcal H(\partial^*E\cap A),
\end{equation}
where $\alpha=\alpha(C_d,C_P,\lambda)>0$, see \cite[Theorem 5.3]{A1} 
and \cite[Theorem 4.6]{AMP}.

The lower and upper approximate limits of a function $u$ on $X$ are defined respectively by
\[
u^{\wedge}(x):
=\sup\left\{t\in\R:\,\lim_{r\to 0}\frac{\mu(\{u<t\}\cap B(x,r))}{\mu(B(x,r))}=0\right\}
\]
and
\[
u^{\vee}(x):
=\inf\left\{t\in\R:\,\lim_{r\to 0}\frac{\mu(\{u>t\}\cap B(x,r))}{\mu(B(x,r))}=0\right\}.
\]
It is straightforward to show that $u^{\wedge}$ and $u^{\vee}$ are Borel functions.

We understand $\BV$ functions to be $\mu$-equivalence classes.
For example, in the coarea formula \eqref{eq:coarea}, each $\{u>t\}$ is precisely speaking not a set
but a $\mu$-equivalence class of sets.
On the other hand, the pointwise
representatives $u^{\wedge}$ and $u^{\vee}$ are defined at every point.
From Lebesgue's differentiation theorem (see e.g. \cite[Chapter 1]{Hei}) it follows that
$u^{\wedge}=u^{\vee}=u$ a.e.

\section{BV functions with zero boundary values}

In this section we define and study a class of $\BV$ functions with zero boundary values.

First we gather some results that we will need.
The following result is well known and proved for sets of finite perimeter in \cite{M},
but we recite the proof of the more general case here.
\begin{lemma}\label{lem:BV functions form algebra}
Let $\Om\subset X$ be an open set and let $u,v\in L^1_{\loc}(\Om)$. Then
\[
\Vert D\min\{u,v\}\Vert(\Om)+\Vert D\max\{u,v\}\Vert(\Om)\le
\Vert Du\Vert(\Om)+\Vert Dv\Vert(\Om).
\]
\end{lemma}
\begin{proof}
We can assume that the right-hand side is finite. Take sequences
of functions $(u_i),(v_i)\subset\liploc(\Om)$
such that $u_i\to u$ and $v_i\to v$ in $L^1_{\loc}(\Om)$, and
\[
\lim_{i\to\infty}\int_{\Om}g_{u_i}\,d\mu= \Vert Du\Vert(\Om)\quad\textrm{and}\quad
\lim_{i\to\infty}\int_{\Om}g_{v_i}\,d\mu= \Vert Dv\Vert(\Om).
\]
By \cite[Corollary 2.20]{BB}, we have
\[
g_{\min\{u_i,v_i\}}=g_{u_i}\ch_{\{u_i\le v_i\}}+g_{v_i}\ch_{\{u_i> v_i\}},\quad\
g_{\max\{u_i,v_i\}}=g_{u_i}\ch_{\{u_i> v_i\}}+g_{v_i}\ch_{\{u_i\le v_i\}}
\]
in $\Om$.
Since also $\min\{u_i,v_i\}\to \min\{u,v\}$ and
$\max\{u_i,v_i\}\to \max\{u,v\}$ in $L^1_{\loc}(\Om)$, we get
\begin{align*}
&\Vert D\min\{u,v\}\Vert(\Om)+\Vert D\max\{u,v\}\Vert(\Om)\\
&\qquad\qquad\le \liminf_{i\to\infty}\int_{\Om} g_{\min\{u_i,v_i\}}\,d\mu+
\liminf_{i\to\infty}\int_{\Om} g_{\max\{u_i,v_i\}}\,d\mu\\
&\qquad\qquad\le \liminf_{i\to\infty}\left(\int_{\Om} g_{u_i}\,d\mu+
\int_{\Om} g_{v_i}\,d\mu\right)\\
&\qquad\qquad= \Vert Du\Vert(\Om)+\Vert Dv\Vert(\Om).
\end{align*}
\end{proof}

Moreover, for any $u,v\in L^1_{\loc}(\Om)$, it is straightforward to show that
\begin{equation}\label{eq:BV functions form vector space}
\Vert D(u+v)\Vert(\Om)\le \Vert Du\Vert(\Om)+\Vert Dv\Vert(\Om).
\end{equation}

Since Lipschitz functions are dense in $N^{1,1}(X)$, see \cite[Theorem 5.1]{BB}, it follows that
\begin{equation}\label{eq:Sobolev subclass BV}
N^{1,1}(X)\subset \BV(X)\quad \textrm{with}\quad \Vert Du\Vert(X)\le \int_X g_u\,d\mu\ \ 
\textrm{for every }u\in N^{1,1}(X).
\end{equation}
Recall that we interpret Newton-Sobolev functions to be pointwise defined, whereas
$\BV$ functions are $\mu$-equivalence classes, but nonetheless the inclusion $N^{1,1}(X)\subset \BV(X)$ has a natural interpretation.

The $\BV$-capacity is often convenient due to the following property
not satisfied by the $1$-capacity:
if $A_1\subset A_2\subset \ldots\subset X$, then
\begin{equation}\label{eq:continuity of BVcap}
\capa_{\BV}\left(\bigcup_{j=1}^{\infty}A_j\right)=\lim_{j\to\infty} \capa_{\BV}(A_j),
\end{equation}
see \cite[Theorem 3.4]{HaKi}.
On the other hand, by \cite[Theorem 4.3]{HaKi} we know that for some constant
$C(C_d,C_P,\lambda)\ge 1$ and any
$A\subset X$, we have
\begin{equation}\label{eq:Newtonian and BV capacities are comparable}
\capa_{\BV}(A)\le \capa_1(A)\le C\capa_{\BV}(A).
\end{equation}
By \cite[Theorem 4.3, Theorem 5.1]{HaKi} we know that for $A\subset X$,
\begin{equation}\label{eq:null sets of Hausdorff measure and capacity}
\capa_1(A)=0\quad \textrm{if and only if}\quad \mathcal H(A)=0.
\end{equation}

The following lemma states that a sequence converging in the $\BV$ norm has a
subsequence converging pointwise $\mathcal H$-almost everywhere.

\begin{lemma}\label{lem:norm and pointwise convergence}
Let $u_i,u\in\BV(X)$ with $u_i\to u$ in $\BV(X)$. By passing to a subsequence (not relabeled),
we have $u_i^{\wedge}\to u^{\wedge}$ and $u_i^{\vee}\to u^{\vee}$ $\mathcal H$-a.e.
\end{lemma}
\begin{proof}
By \cite[Lemma 4.2]{LaSh}, for every $\eps>0$ there exists 
$G\subset X$ with $\capa_1(G)<\eps$ such that by passing to a subsequence, if necessary (not relabeled), 
$u_i^{\wedge}\to u^{\wedge}$ and $u_i^{\vee}\to u^{\vee}$ uniformly in $X\setminus G$.
From this it easily follows that we find a subsequence (not relabeled) such that
$u_i^{\wedge}\to u^{\wedge}$ and $u_i^{\vee}\to u^{\vee}$ $1$-q.e., and then
\eqref{eq:null sets of Hausdorff measure and capacity} completes the proof.
\end{proof}

It is a well-known fact that Newton-Sobolev functions are quasicontinuous;
for a proof see \cite[Theorem 1.1]{BBS} or \cite[Theorem 5.29]{BB}.

\begin{theorem}\label{thm:quasicontinuity}
Let $1\le p<\infty$, let $u\in N^{1,p}(X)$, and let $\eps>0$.
Then there exists an open set $G\subset X$ with $\capa_p(G)<\eps$
such that $u|_{X\setminus G}$ is real-valued continuous.
\end{theorem}

In this paper we will rely heavily on the fact that $\BV$ functions have the following
quasi-semicontinuity property, which was first proved in the Euclidean setting in
\cite[Theorem 2.5]{CDLP}. Since we understand $\BV$ functions to be $\mu$-equivalence classes,
we need to consider the representatives $u^{\wedge}$ and $u^{\vee}$ when studying continuity
properties.

\begin{proposition}\label{prop:quasisemicontinuity}
Let $u\in\BV(X)$ and let $\eps>0$. Then there exists an open set $G\subset X$ with $\capa_1(G)<\eps$
such that $u^{\wedge}|_{X\setminus G}$ is real-valued lower semicontinuous and $u^{\vee}|_{X\setminus G}$ is real-valued upper semicontinuous.
\end{proposition}
\begin{proof}
This follows from \cite[Theorem 1.1]{LaSh}.
\end{proof}

The following fact clarifies the relationship between the different pointwise representatives.

\begin{proposition}\label{prop:Lebesgue points for Sobolev functions}
Let $u\in N^{1,1}(X)$. Then $u=u^{\wedge}=u^{\vee}$ $\mathcal H$-a.e.
\end{proposition}
\begin{proof}
We know that $u$ has Lebesgue points $1$-q.e., that is,
\[
\lim_{r\to 0}\,\vint{B(x,r)}|u-u(x)|\,d\mu=0
\]
for $1$-q.e. $x\in X$, see \cite[Theorem 4.1, Remark 4.2]{KKST2}
(note that in \cite{KKST2} it is assumed that $\mu(X)=\infty$,
but this assumption can be avoided by using \cite[Lemma 3.1]{Mak}
instead of \cite[Theorem 3.1]{KKST2} in the proof of the Lebesgue point theorem).
It follows that $u(x)=u^{\wedge}(x)=u^{\vee}(x)$ for such $x$, and then
\eqref{eq:null sets of Hausdorff measure and capacity} completes the proof.
\end{proof}

Now we turn our attention to defining the class of $\BV$ functions with zero boundary values.
We recall that the Newton-Sobolev class with zero boundary values $N^{1,1}_0(D)$ consists of
the restrictions to $D$ of those functions $u\in N^{1,1}(X)$ with $u=0$ $1$-q.e. on $X\setminus D$, or equivalently
$\mathcal H$-a.e. on $X\setminus D$.
When dealing with $\BV$ functions, we need to consider both representatives $u^{\wedge}$ and $u^{\vee}$,
and thus we give the following definition.

\begin{definition}\label{def:BV functions with zero boundary values}
Let $D\subset X$.
We let
\[
\BV_0(D):=
\left\{u|_{D}:\,u\in\BV(X),\ u^{\wedge}(x)=u^{\vee}(x)=0\textrm{ for }\mathcal H\textrm{-a.e. }x\in X\setminus D\right\}.
\]
\end{definition}

Since $\BV(X)$ consists of $\mu$-equivalence classes of functions on $X$,
$\BV_0(D)$ consists of $\mu$-equivalence classes of functions on $D$.
For an open set $\Om\subset X$, the class $\BV_0(\Om)$ is a subclass of $\BV(\Om)$
(note that we have defined the class $\BV(\Om)$ only for open $\Om\subset X$).
If $u\in\BV_0(D)$ and $u=v|_D=w|_D$ for $v,w\in\BV(X)$
with $v^{\wedge}=v^{\vee}=w^{\wedge}=w^{\vee}=0$ $\mathcal H$-a.e. on $X\setminus D$,
then by Lebesgue's differentiation theorem, necessarily
$v=0=w$ $\mu$-a.e. on $X\setminus D$, and so $v$ and $w$ are the same $\BV$ function.
Thus for any $D\subset X$, the class $\BV_0(D)$ can also be understood to be a subclass of $\BV(X)$.
Most of the time we will in fact, without further notice, understand functions in $\BV_0(D)$
to be defined on the whole space.

Note that for $u\in N^{1,1}(X)$, requiring that $u=0$ $1$-q.e. on $X\setminus D$
is equivalent to requiring that $u^{\wedge}=u^{\vee}=0$ $\mathcal H$-a.e. on
$X\setminus D$,
due to \eqref{eq:null sets of Hausdorff measure and capacity} and
Proposition \ref{prop:Lebesgue points for Sobolev functions}. Thus our
definition of $\BV_0(D)$ is a close analog
of the definition of $N_0^{1,1}(D)$. In fact, since $N^{1,1}(X)\subset \BV(X)$ (recall \eqref{eq:Sobolev subclass BV}), we always have
\begin{equation}\label{eq:Newtonian zero class contained in BV zero}
N_0^{1,1}(D)\subset \BV_0(D).
\end{equation}

\begin{remark}\label{rem:definitions of zero boundary values}
Other definitions of $\BV_0(\Om)$ have been given in previous works, always for open $\Om\subset X$.
In \cite{HKL} the class was defined by requiring that $u=0$ on $X\setminus \Om$ (that is, 
$u=0$ $\mu$-a.e. on $X\setminus \Om$). This definition is convenient when solving Dirichlet
problems,
because the condition persists under $L^1$-limits.

By contrast, when considering functions of least gradient, i.e. \emph{local}
minimizers of the total variation in an open set $\Om$,
it is natural to consider test functions $\varphi\in\BV(\Om)$ satisfying
\[
\lim_{r\to 0}\,\vint{B(x,r)\cap\Om}|\varphi|\,d\mu=0\quad\textrm{for }\mathcal H\textrm{-a.e }x\in \partial\Om
\]
or
\[
\lim_{r\to 0}\,\frac{1}{\mu(B(x,r))}\int_{B(x,r)\cap\Om}|\varphi|\,d\mu=0\quad\textrm{for }\mathcal H\textrm{-a.e }x\in \partial\Om,
\]
see \cite[Section 9]{KLLS}.
The latter condition is very close to that of Definition
\ref{def:BV functions with zero boundary values}, but
here we are not assuming the function $\varphi$ to be in the class $\BV(X)$, only in $\BV(\Om)$.

For $1<p<\infty$,  there seems to be no ambiguity in how the class of Newton-Sobolev
functions with zero boundary values
ought to be defined, because the class $N_0^{1,p}(D)$ is closed under $L^p$-limits
of sequences that are bounded in the $\Vert \cdot\Vert_{N^{1,p}(X)}$-norm
(up to a choice of $\mu$-representative),
which is what one needs in the calculus of variations.
Our current definition of $\BV_0(D)$ does not have the same property,
but it is motivated by the fact that it is otherwise a close analog of
the definition of $N^{1,p}_0(D)$.
In the case $p>1$, it is fruitful to consider the class $N^{1,p}_0(D)$
for \emph{finely open} sets $D$,
e.g. when constructing \emph{p-strict subsets} with the help of a \emph{Cartan property},
see \cite[Lemma 3.3]{BBL-SS}.
We expect the class $\BV_0(D)$ to be similarly useful when extending these concepts to the
case $p=1$ in future work, see \cite{L3,L-FC,L-WC} for results so far.
\end{remark}

\begin{proposition}\label{prop:variation measure concentration for BV0}
Let $D\subset X$ and let $u\in\BV_0(D)$. Then $\Vert Du\Vert(X\setminus D)=0$.
\end{proposition}
\begin{proof}
Let $x\in X\setminus D$ with $u^{\wedge}(x)=u^{\vee}(x)=0$.
It follows in a straightforward manner from the definitions that $x\notin \partial^*\{u>t\}$
for all $t\neq 0$;
recall the definition of the measure-theoretic boundary from \eqref{eq:measure theoretic boundary}.
By combining the coarea formula \eqref{eq:coarea}
and \eqref{eq:def of theta}, it is easy to show that $\Vert Du\Vert$
is absolutely continuous with respect to $\mathcal H$. By using this fact, the coarea formula
in the Borel set $\{u^{\wedge}=0\}\cap \{u^{\vee}=0\}$, and again \eqref{eq:def of theta}, we get
\begin{align*}
\Vert Du\Vert(X\setminus D)
&\le \Vert Du\Vert(\{u^{\wedge}=0\}\cap \{u^{\vee}=0\})\\
&=\int_{-\infty}^{\infty}P(\{u>t\},\{u^{\wedge}=0\}\cap \{u^{\vee}=0\})\,dt\\
&\le C_d \int_{-\infty}^{\infty}\mathcal H(\partial^*\{u>t\}\cap \{u^{\wedge}=0\}\cap \{u^{\vee}=0\})\,dt\\
&=0.
\end{align*}
\end{proof}

By the above proposition, it is natural to equip the space $\BV_0(D)$ with the norm
$\Vert \cdot\Vert_{\BV(X)}$.

It is well known that $\BV(X)$ is a Banach space.
The following proposition states that so is $\BV_0(D)$.

\begin{proposition}\label{prop:BV0 closed subspace}
Let $D\subset X$. Then $\BV_0(D)$ is a closed subspace of $\BV(X)$.
\end{proposition}
\begin{proof}
It is easy to check that $\BV_0(D)$ is a vector space.
Consider a sequence $(u_i)\subset\BV_0(D)$ with $u_i\to u$ in $\BV(X)$. Then it follows from
Lemma \ref{lem:norm and pointwise convergence} that also $u\in\BV_0(D)$.
\end{proof}

Besides $\BV$ functions with zero boundary values, we wish to consider compactly supported $\BV$
functions.
The support of a function $u$ on $X$ is the closed set
\[
\supp u:=\{x\in X:\,\mu(B(x,r)\cap\{u\neq 0\})>0\ \textrm{for all }r>0\}.
\]
Moreover, the positive and negative parts of a function $u$ are
$u_+:=\max\{u,0\}$ and $u_-:=-\min\{u,0\}$.

In the following theorem, we show that $\BV_0(D)$ is the closure of compactly supported functions in the
$\BV$ norm. A similar result has been given previously in \cite[Theorem 6.9]{LaSh2}, but only for open
$D\subset X$, and with additional assumptions either on the space or on the boundary of $D$.

\begin{theorem}\label{thm:characterization of BV function with zero bdry values}
Let $D \subset X$ and let $u\in\BV(X)$.
Then the following are equivalent:
\begin{enumerate}[{(1)}]
\item $u\in \BV_0(D)$.
\item There exists a sequence $(u_k)\subset\BV(X)$ such that $\supp u_k$ is a compact subset of $D$ for each $k\in\N$, and
$u_k\to u$ in $\BV(X)$.
\end{enumerate}
\end{theorem}

\begin{proof}
\hfill\\
\noindent $(1)\implies(2)$: Fix $x_0\in X$
and let $\eta_j(x):=(1-\dist(x,B(x_0,j)))_+$, $j\in\N$, so that each $\eta_j$ is a
$1$-Lipschitz function with $\eta_j=1$ on $B(x_0,j)$ and $\eta_j=0$ outside $B(x_0,j+1)$.
By a suitable Leibniz rule, see \cite[Lemma 3.2]{HKLS}, we have
\[
\Vert D(\eta_j u-u)\Vert(X)\le \Vert Du\Vert(X\setminus B(x_0,j))+\int_{X\setminus B(x_0,j)}|u|\,d\mu\to 0
\]
as $j\to \infty$. Thus we can assume that $u$ is compactly supported (in $X$).
Since $u_+$ and $u_-$ both belong to $\BV_0(D)$ and $u=u_+-u_-$,
we can assume that $u\ge 0$.
Finally, by using the coarea formula \eqref{eq:coarea} it is easy to check that
$\Vert \min\{u,j\}-u\Vert_{\BV(X)}\to 0$ as $j\to\infty$, and so we can also assume that
$u$ is bounded.

Note that for $\eps>0$, by the coarea formula
\begin{align*}
\Vert D(u-(u-\eps)_+)\Vert (X)
&=\Vert D\min\{u,\eps\}\Vert (X)\\
&=\int_{-\infty}^{\infty}P(\min\{u,\eps\}>t\},X)\,dt\\
&=\int_{0}^{\eps}P(\{u>t\},X)\,dt\\
&\to 0\quad\textrm{as }\eps\to 0.
\end{align*}
Clearly also $(u-\eps)_+\to u$ in $L^1(X)$ as $\eps\to 0$.
Fix $\eps>0$.
By Proposition \ref{prop:quasisemicontinuity}
there exist open sets $G_j\subset X$ such that $\capa_1(G_j)\to 0$ and $u^{\vee}|_{X\setminus G_j}$ is
upper semicontinuous for each $j\in\N$.
Since $\mathcal H(\{u^{\vee}>0\}\setminus D)=0$ and thus $\capa_1(\{u^{\vee}>0\}\setminus D)=0$
by \eqref{eq:null sets of Hausdorff measure and capacity}, we can assume that
$\{u^{\vee}>0\}\setminus D\subset G_j$ for each $j\in\N$ (recall also that $\capa_1$ is an outer
capacity).
For any fixed $j\in\N$, since $\{u^{\vee}<\eps\}$ is an open set in the subspace topology of $X\setminus G_j$,
there exists an open set $W\subset X$ such that
\[
W\setminus G_j=\{u^{\vee}<\eps\}\setminus G_j,
\]
and thus $V_j:=\{u^{\vee}<\eps\}\cup G_j=W\cup G_j$ is an open set.
Note that $X\setminus D\subset V_j$, and $X\setminus V_j$ is bounded by the fact that $u$ has
compact support in $X$, so in conclusion $X\setminus V_j$ is a bounded subset of $D$ (in fact, compact).

There exist functions
$w_j\in N^{1,1}(X)$ such that $0\le w_j\le 1$ on $X$, $w_j= 1$ on $G_j$,
and $\Vert w_j\Vert_{N^{1,1}(X)}\to 0$.
Then also $\Vert w_j\Vert_{\BV(X)}\to 0$ by \eqref{eq:Sobolev subclass BV}.
By Proposition \ref{lem:norm and pointwise convergence}, we can extract a subsequence
(not relabeled) such that $w_j^{\vee}(x)\to 0$ for $\mathcal H$-a.e.
$x\in X$.
Let $u_{\eps,j}:=(1-w_j)(u-\eps)_+$, $j\in\N$. Note that $u_{\eps,j}\ge 0$.
Clearly $u_{\eps,j}=0$ on $G_j$ and on $\{u^{\vee}<\eps\}$.
Thus $u_{\eps,j}=0$ in the open set $V_j$, and it follows that
$\supp u_{\eps,j} \subset X\setminus V_{j}$.
Since $X\setminus V_j$ is a bounded subset of $D$, $\supp u_{\eps,j} $ is a compact subset of $D$,
as desired.

Using the Leibniz rule for bounded $\BV$ functions, see
\cite[Proposition 4.2]{KKST3}, we get for some constant $C=C(C_d,C_P,\lambda)\ge 1$
\begin{align*}
\Vert D(u_{\eps,j}-(u-\eps)_+)\Vert(X)
&=\Vert D(w_j(u-\eps)_+)\Vert(X)\\
&\le C\int_X(u-\eps)_+^{\vee}\,d\Vert Dw_j\Vert+C\int_X w_j^{\vee}\,d\Vert D(u-\eps)_+\Vert\\
&\le C\Vert u\Vert_{L^{\infty}(X)}\Vert Dw_j\Vert(X)+C\int_X w_j^{\vee}\,d\Vert D(u-\eps)_+\Vert.
\end{align*}
Here the first term goes to zero since $\Vert w_j\Vert_{\BV(X)}\to 0$, and the second term goes to zero
by Lebesgue's dominated convergence theorem, since $\Vert D(u-\eps)_+\Vert$ is absolutely continuous
with respect to $\mathcal H$. Clearly also
\[
u_{\eps,j}\to(u-\eps)_+\quad\textrm{in }L^1(X)\ \textrm{ as }j\to\infty.
\]
Since we had $(u-\eps)_+\to u$ in $\BV(X)$ as $\eps\to 0$,
by a diagonal argument we can choose numbers $\eps_k\searrow 0$ and indices $j_k\to\infty$ to obtain a sequence $u_k:=u_{\eps_k,j_k}$ such that $u_k\to u$ in $\BV(X)$.\\

\noindent$(2)\implies(1)$: Take a sequence of functions $u_k\in\BV(X)$ such that $\supp u_k$
are compact subsets of $D$ and
$u_k\to u$ in $\BV(X)$.
By Lemma \ref{lem:norm and pointwise convergence} and by passing to a subsequence (not relabeled),
we have $u_k^{\wedge}(x)\to u^{\wedge}(x)$ and $u_k^{\vee}(x)\to u^{\vee}(x)$ for
$\mathcal H$-a.e. $x\in X$. Since clearly $u_k^{\wedge}=u_k^{\vee}=0$ on
$X\setminus \supp u_k\supset X\setminus D$,
also $u^{\wedge}(x)=u^{\vee}(x)=0$ for $\mathcal H$-a.e. $x\in X\setminus D$,
and so $u\in\BV_0(D)$.
\end{proof}

\begin{remark}
The proof of the implication $(1)\implies(2)$ essentially follows along the lines of
the proof of an analogous result for Newton-Sobolev functions given in \cite[Section 5.4]{BB},
but instead of the quasicontinuity of Newton-Sobolev functions
we use the quasi-semicontinuity of the representative $u^{\vee}$.
\end{remark}

\begin{lemma}\label{lem:extension of compactly supported function}
Let $\Om\subset X$ be an open set and let $u\in\BV(\Om)$ such that $\supp u$
is a compact subset of
$\Om$. Then there exists a sequence $(u_i)\subset\Lip_c(\Om)$ such that $u_i\to u$ in
$L^1(\Om)$ and
\[
\Vert Du\Vert(\Om)=\lim_{i\to\infty}\int_{\Om}g_{u_i}\,d\mu,
\]
where each $g_{u_i}$ is the minimal $1$-weak upper gradient of $u_i$ in $\Om$.
It follows that $u\in\BV_0(\Om)$, and then the above holds also with $\Om$ replaced by $X$.
\end{lemma}

\begin{proof}
The first claim is proved in \cite[Lemma 2.6]{HaSh}. To prove the second claim,
denote by $u,u_i$ also the zero extensions of these functions.
Note that the minimal $1$-weak upper gradient $g_{u_i}$ (now as a function defined on $X$)
is clearly the zero extension of $g_{u_i}$ (as a function defined only on $\Om$),
and so we have
\[
\Vert Du\Vert(X)\le \liminf_{i\to\infty}\int_X g_{u_i}\,d\mu=\liminf_{i\to\infty}\int_{\Om} g_{u_i}\,d\mu=\Vert Du\Vert(\Om).
\]
Thus $u\in\BV(X)$ and then clearly $u\in\BV_0(\Om)$.
\end{proof}

Suppose $(u_i)\subset \liploc(X)$ with $u_i\to u$ in $\BV(X)$.
Then $(u_i)$ is a Cauchy sequence in $\BV(X)$, and by \cite[Remark 4.7]{HKLL},
$(u_i)$ is a Cauchy sequence also in $N^{1,1}(X)$. Since $N^{1,1}(X)/\sim$
is a Banach space with the equivalence relation
$u\sim v$ if $\Vert u-v\Vert_{N^{1,1}(X)}=0$, see \cite[Theorem 1.71]{BB},
we conclude that $u\in N^{1,1}(X)$. Thus, for an open set $\Om\subset X$,
$\Lip_c(\Om)$ cannot be dense in $\BV_0(\Om)\supsetneq N_0^{1,1}(\Om)$.
On the other hand, compactly supported Lipschitz
functions are dense in $\BV_0(\Om)$ in the following weak sense.

\begin{proposition}\label{prop:weak density of lipschitz functions}
Let $\Om\subset X$ be an open set and let $u\in\BV_0(\Om)$. Then there exists a sequence
$(u_i)\subset \Lip_c(\Om)$ with $u_i\to u$ in $L^1(X)$ (with the understanding that the functions $u_i$
are extended outside $\Om$ by zero) and
\[
\int_{X}g_{u_i}\,d\mu\to\Vert Du\Vert(X)\quad\textrm{as }i\to\infty.
\]
\end{proposition}

\begin{proof}
By Theorem \ref{thm:characterization of BV function with zero bdry values},
we find a sequence $(v_i)\subset\BV(X)$ of functions with compact support in $\Om$ and
$\Vert v_i-u\Vert_{\BV(X)}<1/i$ for each $i\in\N$.
Then by Lemma \ref{lem:extension of compactly supported function}, for each $i\in\N$
we find $u_i\in\Lip_c(\Om)$ with $\Vert u_i-v_i\Vert_{L^1(X)}<1/i$ and
\[
\left|\int_{X} g_{u_i}\,d\mu- \Vert Dv_i\Vert(X)\right|< 1/i.
\]
We conclude that $\Vert u_i-u\Vert_{L^1(X)}<2/i$ and
\[
\left|\int_{X} g_{u_i}\,d\mu- \Vert Du\Vert(X)\right|< 2/i.
\]
\end{proof}

Now we can also show the following result, the analog of which is well known for Newton-Sobolev
functions, see \cite[Lemma 2.37]{BB}.

\begin{proposition}
Let $\Om\subset X$ be an open set and let  $u\in \BV(\Om)$ and $v,w\in \BV_0(\Om)$ such that
$v\le u\le w$ in $\Om$.
Then $u\in\BV_0(\Om)$.
\end{proposition}
\begin{proof}
By subtracting $v$ from all terms and observing that $u\in \BV_0(\Om)$ if and only if $u-v\in \BV_0(\Om)$,
we can assume that $v\equiv 0$.
Denote the zero extension of $u$ outside $\Om$ by $u_0$. 
By Theorem \ref{thm:characterization of BV function with zero bdry values}, we find a sequence of
nonnegative functions $(w_k)\subset\BV(X)$ compactly supported in $\Om$ with
$w_k\to w$ in $\BV(X)$ (the nonnegativity actually follows from the proof, or alternatively by truncation).
Then $\varphi_k:=\min\{w_k,u_0\}\in\BV(\Om)$ by Lemma \ref{lem:BV functions form algebra},
and $\varphi_k\in \BV(X)$ by Lemma
\ref{lem:extension of compactly supported function}, for each $k\in\N$.
Moreover, $\varphi_k\to u_0$ in $L^1(X)$,
and since each $\varphi_k$ has compact support in $\Om$,
\begin{align*}
\liminf_{k\to\infty}\Vert D\varphi_k\Vert(X)
&=\liminf_{k\to\infty}\Vert D\varphi_k\Vert(\Om)\\
&\le \liminf_{k\to\infty}\Vert Dw_k\Vert(\Om)
+\Vert Du_0\Vert(\Om)\quad\textrm{by Lemma }\ref{lem:BV functions form algebra}\\
&= \Vert Dw\Vert(\Om)+\Vert Du\Vert(\Om).
\end{align*}
Thus by the lower semicontinuity of the total variation with respect to $L^1$-convergence, $u_0\in\BV(X)$. Moreover,
$u_0^{\vee}(x)\le w^{\vee}(x)=0$ for $\mathcal H$-a.e. $x\in X\setminus\Om$,
and obviously $u_0^{\wedge}(x)\ge 0$ for all $x\in X\setminus\Om$,
guaranteeing that $u_0^{\wedge}=u_0^{\vee}=0$ $\mathcal H$-a.e. in $X\setminus\Om$.
\end{proof}

\section{The variational $1$-capacity}\label{sec:variational capacity}

In this section we study the variational (Newton-Sobolev) $1$-capacity and its Lipschitz and
$\BV$ analogs. Utilizing the results of
the previous section, we show that each of these is an outer capacity, and that the
capacities are equal for certain sets.

\begin{definition}
Let $A\subset D\subset X$ be arbitrary sets.
We define the variational (Newton-Sobolev) $1$-capacity by
\[
\rcapa_1(A,D):=\inf \int_X g_u\,d\mu,
\]
where the infimum is taken over functions $u\in N^{1,1}_0(D)$ such that $u\ge 1$ on $A$.

We define the variational Lipschitz $1$-capacity by
\[
\rcapa_{\mathrm{lip}}(A,D):=\inf \int_{X}g_u\,d\mu,
\]
where the infimum is taken over functions $u\in N^{1,1}_0(D)\cap \liploc(X)$
such that $u\ge 1$ on $A$.

Finally, we define the variational $\BV$-capacity by
\[
\rcapa_{\BV}(A,D):=\inf \Vert Du\Vert(X),
\]
where the infimum is taken over functions $u\in\BV_0(D)$ such that $u^{\wedge}\ge 1$ $\mathcal H$-almost everywhere on $A$.

In each case, we say that the functions $u$ over which we take the infimum are admissible (test) functions
for the capacity in question.
\end{definition}

Again, $g_u$ always denotes the minimal $1$-weak upper gradient of $u$.
Recall that we understand Newton-Sobolev functions to be defined at every point, but in
the definition of $\rcapa_1(A,D)$ we can equivalently
require $u\ge 1$ $1$-q.e. on $A$, by \eqref{eq:quasieverywhere equivalence classes}.
On the other hand, perturbing the representatives $u^{\wedge}$ and $u^{\vee}$ even at a single point
requires perturbing the function $u$ in a set of positive $\mu$-measure.
In each definition, we see by truncation that it is enough to consider test functions $0\le u\le 1$,
and then the conditions $u\ge 1$ and $u^{\wedge}\ge 1$ are replaced by $u= 1$ and $u^{\wedge}= 1$,
respectively.

Our definition of $\rcapa_1$ is the same as the one given in \cite{BB-cap},
where the variational $p$-capacity was studied for all $1\le p<\infty$.
In our definition of $\rcapa_{\BV}$, we have then mimicked the definition of $\rcapa_1$
as closely as possible --- note that for $u\in N_0^{1,1}(D)$, requiring that
$u= 1$ $1$-q.e. on $A$ is equivalent to requiring that $u^{\wedge}= 1$ $\mathcal H$-a.e. on $A$, due to \eqref{eq:null sets of Hausdorff measure and capacity} and
Proposition \ref{prop:Lebesgue points for Sobolev functions}.
Since $N_0^{1,1}(D)\subset \BV_0(D)$ with $\Vert Du\Vert(X)\le \int_X g_u\,d\mu$
for every $u\in N_0^{1,1}(D)$ (recall \eqref{eq:Sobolev subclass BV} and
\eqref{eq:Newtonian zero class contained in BV zero}), we conclude that always
$\rcapa_{\BV}(A,D)\le \rcapa_1(A,D)$.
Clearly we also always have $\rcapa_1(A,D)\le \rcapa_{\mathrm{lip}}(A,D)$.
In Theorem \ref{thm:BV and lip caps are equal} and Example \ref{ex:weak density fails} below
we investigate when equalities hold.

Also other definitions of $\rcapa_{\BV}$ have been given in the literature.
In \cite{HaSh}, given an open set $\Om\subset X$ and a compact set $K\subset\Om$,
$\rcapa_{\BV}(K,\Om)$ was defined by considering $\BV$ test functions that are compactly supported in  $\Om$ and take the value $1$ in a neighborhood of $K$. By Theorem \ref{thm:BV and lip caps are equal},
this turns out to agree with our current definition of $\rcapa_{\BV}(K,\Om)$.
For more general sets, however, the definitions can give different results.

\begin{example}\label{ex:first comparisons of capacities}
Let $X=\R^2$ (unweighted), and let $D=[0,1]\times [0,1]$ and $A=\{(0,0)\}$.
Since $\capa_1(A)=0$, also $\rcapa_1(A,D)=0$ (as $u\equiv 0$ satisfies $u= 1$ $1$-q.e. on $A$).
On the other hand, if we defined $\rcapa_{\BV}(A,D)$ by requiring the test functions to take the value $1$ in a \emph{neighborhood} of $A$, as in e.g. \cite{HaSh,KKST-DG},
then we would have $\rcapa_{\BV}(A,D)=\infty$,
since there are no admissible functions.
The same would already happen if we required that $u^{\wedge}=1$ at \emph{every} point in $A$,
instead of $\mathcal H$-a.e. point.  
Our current definition
of the variational $\BV$-capacity has the advantage that the natural inequality
$\rcapa_{\BV}(A,D)\le \rcapa_{1}(A,D)$ always holds.
From this example we also see that it is possible to have $\rcapa_1(A,D)< \rcapa_{\mathrm{lip}}(A,D)$
(the latter being $\infty$).
\end{example}

In \cite{L3}, for an open set $\Om\subset X$ and an arbitrary set $A\subset\Om$, $\rcapa_{\BV}(A,\Om)$ was defined otherwise similarly as here, but the condition
$u\in\BV_0(\Om)$ was replaced by the condition
$u=0$ on $X\setminus \Om$ (meaning that $u=0$ $\mu$-a.e. on $X\setminus \Om$).
This corresponds to the different possible ways of defining the class of $\BV$ functions with
zero boundary values, as discussed earlier.
The advantage of
the definition in \cite{L3} is that in some cases it is possible to prove the existence of capacitary potentials, i.e. admissible functions $u$ that yield the infimum
in the definition of $\rcapa_{\BV}$.

\begin{example}
Let $X=\R$, let
\[
w(x):=
\begin{cases}
1 & \textrm{for }x< 0,\\
1+x & \textrm{for }0\le x\le 1,\\
2 & \textrm{for }1< x,
\end{cases}
\]
and let $d\mu:=w\,d\mathcal L^1$, where $\mathcal L^1$ is the $1$-dimensional Lebesgue measure.
Let $A=(1,2)$ and $D=(0,3)$.
Defining $u_i:=\ch_{(1/i,2)}$, we get
\[
\rcapa_{\BV}(A,D)\le \Vert Du_i\Vert(X)=3+1/i,\quad i\in\N.
\]
Thus $\rcapa_{\BV}(A,D)\le 3$.
Conversely, let $0\le u\le 1$ be an admissible function.
Note that $\mathcal H$ is now comparable to the counting measure, and so
necessarily $u^{\vee}(0)=0$. Thus for some $0<r<1$ we have
\[
\frac{\mathcal L^1(B(0,r)\cap \{u<1/2\})}{\mathcal L^1(B(0,r))}>\frac{1}{2}.
\]
Let $(v_i)\subset\liploc(\R)$ with $v_i\to u$ in $L^1_{\loc}(\R)$ and
$\int_{\R}g_{v_i}\,d\mu\to \Vert Du\Vert(\R)$.
By passing to a subsequence (not relabeled) we have $v_i(x)\to u(x)$ for a.e. $x\in \R$.
Then we have $v_i(x_1)\to 0$ for some $x_1<0$,
$v_i(x_2)\to u^{\vee}(x_2)<1/2$ for some $0<x_2<1$,
$v_i(x_3)\to 1$ for some $1<x_3<2$,
and $v_i(x_4)\to 0$ for some $x_4>3$.
Thus
\begin{align*}
\int_\R g_{v_i}\,d\mu
&\ge \int_{x_1}^{x_2} g_{v_i}\,d\mu+\int_{x_2}^{x_3} g_{v_i}\,d\mu+\int_{x_3}^{x_4} g_{v_i}\,d\mu\\
&\ge |v_i(x_1)-v_i(x_2)|+ w(x_2)|v_i(x_2)-v_i(x_3)|+2|v_i(x_3)-v_i(x_4)|\\
&\to u^{\vee}(x_2)+  w(x_2)(1-u^{\vee}(x_2))+2\quad\textrm{as }i\to\infty\\
&>3
\end{align*}
since $w(x_2)>1$.
Thus $\Vert Du\Vert(\R)>3$, that is, $\rcapa_{\BV}(A,D)= 3$ but no admissible function
gives this infimum.
On the other hand, if we defined $\rcapa_{\BV}(A,D)$ by only requiring that $u=0$ on
$\R\setminus D$, then
the function $\ch_D$ would be admissible and $\rcapa_{\BV}(A,D)=3=\Vert D\ch_D\Vert(\R)$.
The drawback of such a definition is that Theorem \ref{thm:BV and lip caps are equal}
below would no longer hold, see Example \ref{ex:comparison of capacities}.
\end{example}

Now we prove a few simple properties of the variational $\BV$-capacity, the analogs of
which are known for the $1$-capacity $\rcapa_1$, see \cite[Theorem 3.4]{BB-cap}.

\begin{proposition}\label{prop:basic properties of the variational BV-capacity}
The following hold:
\begin{enumerate}[{(1)}]
\item For any $D\subset X$, $\rcapa_{\BV}(\emptyset,D)=0$.
\item If $A_1\subset A_2\subset D$, then $\rcapa_{\BV}(A_1,D)\le \rcapa_{\BV}(A_2,D)$.
\item If $A\subset D_1\subset D_2$, then $\rcapa_{\BV}(A,D_2)\le \rcapa_{\BV}(A,D_1)$.
\item If $A_1,A_2\subset D\subset X$,
\[
\rcapa_{\BV}(A_1\cap A_2,D)+\rcapa_{\BV}(A_1\cup A_2,D)\le \rcapa_{\BV}(A_1,D)+\rcapa_{\BV}(A_2,D).
\]
\end{enumerate}
\end{proposition}

\begin{proof}
\hfill\\
\noindent$(1)$--$(3)$: These statements are trivial.\\

\noindent$(4)$: We can assume that the right-hand side is finite. Fix $\eps>0$.
Take $u_j\in \BV_0(D)$ with $0\le u_j\le 1$, $u_j^{\wedge}=1$ on $A_j$, and
$\Vert Du_j\Vert(X)< \rcapa_{\BV}(A_j,D)+\eps$,
$j=1,2$. Let $v:=\min\{u_1,u_2\}$ and $w:=\max\{u_1,u_2\}$.
By Lemma \ref{lem:BV functions form algebra} we have
\[
\Vert Dv\Vert(X)+\Vert Dw\Vert(X)\le \Vert Du_1\Vert(X)+\Vert Du_2\Vert(X),
\]
and so $v,w\in\BV(X)$.
Clearly $w^{\wedge}=1$ on $A_1\cup A_2$. To verify that $v^{\wedge}=1$ on $A_1\cap A_2$,
we note that for any $x\in A_1\cap A_2$ and any $\delta>0$,
\begin{align*}
&\limsup_{r\to 0}\frac{\mu(\{v<1-\delta\}\cap B(x,r)))}{\mu(B(x,r))}\\
&\ \ \le
\limsup_{r\to 0}\frac{\mu(\{u_1<1-\delta\}\cap B(x,r))}{\mu(B(x,r))}
+\limsup_{r\to 0}\frac{\mu(\{u_2<1-\delta\}\cap B(x,r))}{\mu(B(x,r))}\\
&\ \  = 0
\end{align*}
by the fact that $u_1^{\wedge}(x)=u_2^{\wedge}(x)=1$.
Thus $v^{\wedge}(x)\ge 1-\delta$, and by letting $\delta\to 0$ we get $v^{\wedge}(x)= 1$.
Similarly, $v^{\vee}=0=w^{\vee}$ on $X\setminus D$, so that $v,w\in\BV_0(D)$.
Thus
\begin{align*}
&\rcapa_{\BV}(A_1\cap A_2,D)+\rcapa_{\BV}(A_1\cup A_2,D)
\le \Vert Dv\Vert(X)+\Vert Dw\Vert(X)\\
&\qquad\qquad\qquad\qquad\qquad\le \rcapa_{\BV}(A_1,D)+\rcapa_{\BV}(A_2,D)+2\eps.
\end{align*}
Letting $\eps\to 0$ completes the proof.
\end{proof}

Next we show that each of the three capacities we have defined is an outer capacity, in a suitable sense.
First we prove this for the variational (Newton-Sobolev) $1$-capacity. This gives a positive answer to a question posed in \cite{BB-cap},
where the analogous result for $1<p<\infty$ was proved.
In fact, using methods similar to those in \cite{BB-cap},
we give a proof that covers all the cases $1\le p<\infty$;
recall the definition of $\rcapa_p$ from \eqref{eq:def of variational p-capacity}.

We need the following lemma, which is a special case of \cite[Lemma 1.52]{BB}.

\begin{lemma}\label{lem:sup is upper gradient}
Let $u_i\le 1$, $i\in\N$, be functions on $X$ with $p$-weak upper gradients $g_i$.
Let $u:=\sup_{i\in\N} u_i$ and $g:=\sup_{i\in\N} g_i$. Then $g$ is a $p$-weak upper gradient of $u$.
\end{lemma}

\begin{theorem}\label{thm:1-capacity is outer}
Let $1\le p<\infty$ and let $D\subset X$ and $A\subset \inte D$. Then
\[
\rcapa_p(A,D)=\inf_{\substack{V\textrm{ open} \\A\subset V\subset D}}\rcapa_p(V,D).
\]
\end{theorem}
\begin{proof}
One inequality is clear.
To prove the opposite inequality,
we can assume that $\rcapa_p(A,D)<\infty$. Fix $\eps>0$.
Take $u\in N_0^{1,p}(D)$ with $0\le u\le 1$, $u=1$ on $A$,
and $\int_X g_u^p\,d\mu<\rcapa_p(A,D)+\eps$.
For each $j\in\N$, let
\[
D_j:=\{x\in D:\,\dist(x,X\setminus D)>1/j\}
\]
and $A_j:=A\cap D_j$.
Take $j$-Lipschitz functions $\eta_j:=(1-j\dist(\cdot,D_j))_+$,
so that $0\le \eta_j\le 1$ on $X$ and $\eta_j=1$ on $D_j$.
Fix $j\in \N$.
By the quasicontinuity of Newton-Sobolev functions
(recall Theorem \ref{thm:quasicontinuity}),
there exists an open set $G_j\subset X$ with $\capa_p(G_j)^{1/p}<2^{-j}\eps/j$ such that
$u|_{X\setminus G_j}$ is continuous.
Thus there exists an open set $W\subset X$ such that
\[
W\setminus G_j=\{u>1-\eps\}\setminus G_j.
\]
Thus the set $\{u>1-\eps\}\cup G_j=W\cup G_j$ is open,
and then so is
\[
V_j:=(\{u>1-\eps\}\cup G_j)\cap D_j.
\]
Since
$u=1$ on $A_j$,
we conclude that $A_j\subset V_j$.
Take a function $v_j\in N^{1,p}(X)$ with $0\le v_j\le 1$ on $X$, $v_j=1$ on $G_j$,
and $\Vert v_j\Vert_{N^{1,p}(X)} <2^{-j}\eps/j$.
Let $w_j:=v_j \eta_j$. Then
\begin{equation}\label{eq:estimate L1 norm of rhoj}
\left(\int_X w_j^p\,d\mu\right)^{1/p}\le \left(\int_X v_j^p\,d\mu\right)^{1/p}<2^{-j}\eps/j,
\end{equation}
and by the Leibniz rule \cite[Theorem 2.15]{BB},
\begin{align*}
\left(\int_X g_{w_j}^p\,d\mu\right)^{1/p}
&\le \left(\int_X (v_j g_{\eta_j})^p\,d\mu\right)^{1/p}
+\left(\int_X (\eta_j g_{v_j})^p\,d\mu\right)^{1/p}\\
&\le j \left(\int_X v_j^p \,d\mu\right)^{1/p}+\left(\int_X g_{v_j}^p\,d\mu\right)^{1/p}\\
&\le j 2^{-j}\eps/j+2^{-j}\eps/j\\
&\le 2^{-j+1}\eps.
\end{align*}
Now $w_j=1$ on $G_j\cap D_j$, and so $u+w_j>1-\eps$ on $V_j$.
Let $w:=\sup_{j\in\N}w_j$.
Then $u+w>1-\eps$ in the open set $V:=\bigcup_{j=1}^{\infty}V_j$.
Note that $A=\bigcup_{j=1}^{\infty}A_j$ since $A\subset \inte D$, and so $A\subset V$.
We have $w\in L^p(X)$ by \eqref{eq:estimate L1 norm of rhoj}, and by Lemma
\ref{lem:sup is upper gradient}
we know that $g_{w}\le \sup_{j\in\N} g_{w_j}$, so that
\[
\left(\int_X g_{w}^p\,d\mu\right)^{1/p}
\le\sum_{j=1}^{\infty}\left(\int_X g_{w_j}^p\,d\mu\right)^{1/p}\le
\sum_{j=1}^{\infty}2^{-j+1}\eps
= 2\eps.
\]
Clearly $w=0$ on $X\setminus D$, and so we conclude $w\in N_0^{1,p}(D)$.
Then $(u+w)/(1-\eps)\in N_0^{1,p}(D)$ is an admissible function for the set $V$, whence
\begin{align*}
\rcapa_{p}(V,D)^{1/p}
&\le \frac{1}{1-\eps}\left(\int_X g_{u+w}^p\,d\mu\right)^{1/p}\\
&\le \frac{1}{1-\eps}\left(\left(\int_X g_{u}^p\,d\mu\right)^{1/p}
+\left(\int_X g_{w}^p\,d\mu\right)^{1/p}\right)\\
&\le \frac{1}{1-\eps}\left(\left(\rcapa_p(A,D)+\eps\right)^{1/p}+2\eps\right).
\end{align*}
Since $\eps>0$ was arbitrary, we have the result.
\end{proof}

Now we can extend a few other results of \cite{BB-cap} to the case $p=1$.

\begin{proposition}\label{prop:sequence of compact sets}
Let $D\subset X$ and let $K_1\supset K_2\supset \ldots \supset K:=\bigcap_{j=1}^{\infty}K_j$
be compact subsets of $\inte D$. Then
\[
\rcapa_{1}(K,D)=\lim_{j\to\infty}\rcapa_{1}(K_j,D).
\]
\end{proposition}
\begin{proof}
Follow verbatim the proof of \cite[Theorem 4.8]{BB-cap}, except that instead of \cite[Theorem 4.1]{BB-cap},
refer to Theorem \ref{thm:1-capacity is outer}.
\end{proof}

\begin{example}
Let $X=\R$ (unweighted), let $\Om:=(0,2)$, and let $A_j:=(1/j,1)$, $j\in\N$.
Then it is easy to check that
\[
\rcapa_1(A_j,\Om)=2=\rcapa_{\BV}(A_j,\Om)
\]
for all $j\in\N$, but
\[
\rcapa_1\left(\bigcup_{j=1}^{\infty}A_j,\Om\right)=\infty=\rcapa_{\BV}\left(\bigcup_{j=1}^{\infty}A_j,\Om\right),
\]
since there are no admissible functions.
This shows that
\[
\rcapa_1\left(\bigcup_{j=1}^{\infty}A_j,\Om\right)\neq \sup_{\substack{K\textrm{ compact} \\K\subset \bigcup_{j=1}^{\infty}A_j}}
\rcapa_1(K,\Om)
\]
(and similarly for $\rcapa_{\BV}$).
Thus neither $\rcapa_1(\cdot,\Om)$ nor $\rcapa_{\BV}(\cdot,\Om)$ is a \emph{Choquet capacity}, see e.g.
\cite{BB-cap} for more discussion on Choquet capacities.
Note that by contrast, the $\BV$-capacity $\capa_{\BV}$ \emph{is} continuous with respect to increasing
sequences of sets, recall \eqref{eq:continuity of BVcap}, and a Choquet capacity,
see \cite[Corollary 3.8]{HaKi}.
\end{example}

As a small digression, following \cite{BB-cap}, let us define for bounded $D\subset X$ and $A\subset D$ 
\[
\widetilde{\rcapa}_1(A,D):=\inf_{\substack{V\textrm{ relatively open in }D\\ A\subset V\subset D}}\rcapa_1(V,D)
=\inf_{\substack{V\textrm{ open}\\ A\subset V}}\rcapa_1(V\cap D,D).
\]
By Theorem \ref{thm:1-capacity is outer}, clearly
\begin{equation}\label{eq:cap and tildecap coincide}
\rcapa_1(A,D)=\widetilde{\rcapa}_1(A,D)\quad\textrm{for any }
A\subset\inte D.
\end{equation}

\begin{proposition}\label{prop:cap and tildecap}
Let $A\subset D$ be bounded sets. Then $\rcapa_1(A,D)=\widetilde{\rcapa}_1(A,D)$ or
$\widetilde{\rcapa}_1(A,D)=\infty$.
\end{proposition}
\begin{proof}
Follow verbatim the proof of \cite[Proposition 6.5]{BB-cap}, except that instead of \cite[Remark 6.4]{BB-cap},
refer to \eqref{eq:cap and tildecap coincide}.
\end{proof}

\begin{remark}
In Theorem \ref{thm:1-capacity is outer} for $p=1$,
Proposition \ref{prop:sequence of compact sets},
and Proposition \ref{prop:cap and tildecap}, our standing assumptions that
$X$ is complete, $\mu$ is doubling and
the space supports a $(1,1)$-Poincar\'e inequality can be weakened to the assumption
that all functions in $N^{1,1}(X)$ are quasicontinuous, and additionally that $X$ has the
\emph{zero $1$-weak upper gradient property}
in the case of Proposition \ref{prop:cap and tildecap}, see \cite{BB-cap}.
\end{remark}

For the variational Lipschitz $1$-capacity, we obviously have for any $A\subset D\subset X$ that
\[
\rcapa_{\mathrm{lip}}(A,D)=\inf_{\substack{V\textrm{ open} \\A\subset V\subset D}}\rcapa_{\mathrm{lip}}(V,D).
\]
(Of course, both sides may be $+\infty$.)
Next we show that also the variational $\BV$-capacity is an outer capacity,
in the same sense as the
variational (Newton-Sobolev) $1$-capacity.
The proof is almost the same, but instead of the quasicontinuity
of Newton-Sobolev functions we again rely on quasi-semicontinuity, this time of the lower representative $u^{\wedge}$.

\begin{theorem}\label{thm:BV capacity is outer}
Let $D\subset X$ and $A\subset \inte D$. Then
\[
\rcapa_{\BV}(A,D)=\inf_{\substack{V\textrm{ open} \\A\subset V\subset D}}\rcapa_{\BV}(V,D).
\]
\end{theorem}

\begin{proof}
One inequality is clear.
To prove the opposite inequality,
we can assume that $\rcapa_{\BV}(A,D)<\infty$.
Fix $\eps>0$.
Take $u\in \BV_0(D)$ with $0\le u\le 1$, $u^{\wedge}=1$ on $A\setminus N$
for some $\mathcal H$-negligible set $N$,
and $\Vert Du\Vert(X)<\rcapa_{\BV}(A,D)+\eps$.
For each $j\in\N$, let
\[
D_j:=\{x\in D:\,\dist(x,X\setminus D)>1/j\}
\]
and $A_j:=A\cap D_j$.
Take $j$-Lipschitz functions $\eta_j:=(1-j\dist(\cdot,D_j))_+$,
so that $0\le \eta_j\le 1$ on $X$ and $\eta_j=1$ on $D_j$. Fix $j\in \N$.
By Proposition \ref{prop:quasisemicontinuity},
there exists an open set $G_j\subset X$ with $\capa_1(G_j)<2^{-j}\eps/j$ such that
$u^{\wedge}|_{X\setminus G_j}$ is lower semicontinuous.
We can assume that $N\subset G_j$
(recall that $\capa_1$ is an outer capacity).
Thus the set $\{u^{\wedge}>1-\eps\}\cup G_j$ is open,
and then so is
\[
V_j:=(\{u^{\wedge}>1-\eps\}\cup G_j)\cap D_j.
\]
Since
$u^{\wedge}(x)=1$ for every $x\in A_j\setminus G_j$,
we conclude that $A_j\subset V_j$.

Take $v_j\in N^{1,1}(X)$ with $0\le v_j\le 1$ on $X$, $v_j=1$ on $G_j$, and $\Vert v_j\Vert_{N^{1,1}(X)}<2^{-j}\eps/j$.
Let $w_j:=v_j \eta_j$.
Now $w_j^{\wedge}=1$ on $G_j\cap D_j$, and so $(u+w_j)^{\wedge}>1-\eps$ on $V_j$.
Let $w:=\sup_{j\in\N}w_j$.
Then $(u+w)^{\wedge}>1-\eps$ on $V:=\bigcup_{j=1}^{\infty}V_j$.
Note that $A=\bigcup_{j=1}^{\infty}A_j$, and so $A\subset V$.
The function $w$ is the same function as in the proof of Theorem \ref{thm:1-capacity is outer}
(for $p=1$),
and so $w\in N_0^{1,1}(D)\subset \BV_0(D)$, and by \eqref{eq:Sobolev subclass BV},
\[
\Vert Dw\Vert(X)\le\int_X g_{w}\,d\mu\le 2\eps.
\]
Hence $u+w\in \BV_0(D)$, and
\begin{align*}
\rcapa_{\BV}(V,D)
&\le \frac{1}{1-\eps}\Vert D(u+w)\Vert(X)\\
&\le \frac{1}{1-\eps}(\Vert Du\Vert(X)+\Vert Dw\Vert(X))\quad\textrm{by }\eqref{eq:BV functions form vector space}\\
&\le \frac{1}{1-\eps}(\Vert Du\Vert(X)+2\eps)\\
&\le\frac{1}{1-\eps}(\rcapa_{\BV}(A,D)+3\eps).
\end{align*}
Since $\eps>0$ was arbitrary, we have the result.
\end{proof}

For the $\BV$-capacity $\capa_{\BV}$, which we have essentially \emph{defined} as
an outer capacity, we can analogously (and much more easily) show the following;
see also \cite[Section 2]{CDLP}
(which uses \cite[Section 4]{FedZie}) for a corresponding result in the Euclidean setting.
\begin{proposition}\label{prop:BV-capacity outer capacity}
For any $A\subset X$,
\[
\capa_{\BV}(A)=\inf \Vert u\Vert_{\BV(X)},
\]
where the infimum is taken over all $u\in\BV(X)$ with $u^{\wedge}(x)\ge 1$ for
$\mathcal H$-a.e. $x\in A$.
\end{proposition}

\begin{proof}
One inequality is clear. To prove the opposite inequality,
fix $A\subset X$
and denote the infimum on the right-hand side by $\beta$;
we can assume that $\beta<\infty$.
Fix $\eps>0$ and
take $u\in\BV(X)$ such that $u^{\wedge}(x)\ge 1$ for every $x\in A\setminus N$ for some
$\mathcal H$-negligible set $N$, and
$\Vert u\Vert_{\BV(X)}<\beta+\eps$.
By Proposition \ref{prop:quasisemicontinuity}, we find an open set $G\subset X$ such that
$\capa_1(G)<\eps$ and
$u^{\wedge}|_{X\setminus G}$ is lower semicontinuous, and we can assume that $N\subset G$.
Take $w\in N^{1,1}(X)$ such that $w\ge 1$ on $G$ and $\Vert w\Vert_{N^{1,1}(X)}<\eps$.
By \eqref{eq:Sobolev subclass BV}, $w\in \BV(X)$ with $\Vert w\Vert_{\BV(X)}<\eps$.
Now $u+w> 1-\eps$ on $\{u^{\wedge}>1-\eps\}\cup G$, which is an open set
containing $A$, and so
\begin{align*}
\capa_{\BV}(A)
\le \frac{\Vert u+w\Vert_{\BV(X)}}{1-\eps}
&\le \frac{\Vert u\Vert_{\BV(X)}+\Vert w\Vert_{\BV(X)}}{1-\eps}\\
&\le \frac{\Vert u\Vert_{\BV(X)}+\eps}{1-\eps}\le \frac{\beta+2\eps}{1-\eps}.
\end{align*}
Letting $\eps\to 0$, we get the result.
\end{proof}

Now we can prove Maz'ya-type inequalities for $\BV$ functions.
We adapt the proof of \cite[Theorem~5.53]{BB}, where such inequalities are given for
Newton-Sobolev functions (the inequalities were originally proven in the
Euclidean setting in \cite{Maz1}; see also \cite[Theorem~10.1.2]{Maz2}).
In the following, given a ball $B=B(x,r)$ and $\beta>0$, we use the abbreviation
$\beta B:=B(x,\beta r)$.
Moreover, recall the definition of the exponent $Q>1$ from \eqref{eq:homogenous dimension},
and the constants $C_P$ and $C_{SP}$ from the
Poincar\'e and Sobolev-Poincar\'e inequalities. 

\begin{theorem}\label{thm:mazya type inequality}
	Let $u\in\BV(X)$, let $S:= \{u^{\wedge}=u^{\vee}=0\}$, and let
	$B=B(x,r)$ for some $x\in X$ and $r>0$.
	Then we have
	\begin{equation}\label{eq:first mazya inequality}
	\left(\,\vint{2B}|u|^{Q/(Q-1)}\,d\mu\right)^{(Q-1)/Q}\le
	\frac{3(C_P+C_{SP})(r+1)}{\capa_{\BV}(B \cap S)}\Vert Du\Vert(4\lambda B)
	\end{equation}
	and
	\[
	\left(\,\vint{2B}|u|^{Q/(Q-1)}\,d\mu\right)^{(Q-1)/Q}\le
	\frac{3(C_P+C_{SP})}{\text{\rm cap}_{\BV}(B \cap S, 2B)}\Vert Du\Vert(4\lambda B),
	\]
	if the denominators are nonzero.
\end{theorem}

\begin{proof}
	Let $q:=Q/(Q-1)$.
	First assume that $u$ is nonnegative.
	Let
	\[
	a:=\left(\,\vint{2B}u^q\,d\mu\right)^{1/q}.
	\]
	We can clearly assume that $a>0$.
	Take a $1/r$-Lipschitz function $0\le\eta\le 1$ with $\eta=1$ on $B$ and 
	$\eta=0$ on $X\setminus 2B$, and then let
	$v:=\eta(1-u/a)$. Now $v\in\BV(X)$
	(this actually follows from \eqref{eq:capacity estimate} below)
	with $v^{\wedge}\le 0$, $v^{\vee}\le 0$ on $X\setminus 2B$ 
	and $v^{\wedge}=v^{\vee}=1$ on $B\cap S$. By Proposition \ref{prop:BV-capacity outer capacity}
	and a suitable Leibniz rule, see \cite[Lemma 3.2]{HKLS}, we get
	\begin{equation}\label{eq:capacity estimate}
		\begin{split}
			\capa_{\BV}(B\cap S)
			&\le \int_X |v|\,d\mu+\Vert Dv\Vert(X) \\
			&\le \int_X |v|\,d\mu+\frac{1}{a}\left(\int_X \eta\,d\Vert Du\Vert+\int_{X}g_{\eta}|u-a|\,d\mu\right)\\
			&\le \frac{1}{a}\int_{2B}|u-a|\,d\mu+\frac{1}{a}\left(\Vert Du\Vert(2B)+\int_{2B}g_{\eta}|u-a|\,d\mu\right)\\
			&\le \frac{1+r^{-1}}{a}\int_{2B}|u-a|\,d\mu+\frac{1}{a}\Vert Du\Vert(2B).
		\end{split}
	\end{equation}
	To estimate the first term, we write
	\begin{equation}\label{eq:relative capacity, using triangle inequality}
	\begin{split}
	\int_{2B}|u-a|\,d\mu
	&\le \int_{2B}|u-u_{2B}|\,d\mu+|u_{2B}-a|\mu(2B)\\
	&\le 2C_Pr\Vert Du\Vert(2\lambda B)+|u_{2B}-a|\mu(2B)
	\end{split}
	\end{equation}
	by the Poincar\'e inequality.
	Here the second term can be estimated by
	\begin{align*}
		|a-u_{2B}|\,&\mu(2B)^{1/q}
		=|\Vert u\Vert_{L^q(2B)}-\Vert u_{2B}\Vert_{L^q(2B)}|\\
		&\qquad\qquad\le \Vert u-u_{2B}\Vert_{L^q(2B)}\\
		&\qquad\qquad= \left(\,\vint{2B}|u-u_{2B}|^q\,d\mu\right)^{1/q}\mu(2B)^{1/q}\\
		&\qquad\qquad\le 2C_{SP}r\frac{\Vert Du\Vert (4\lambda B)}{\mu(4\lambda B)}\mu(2B)^{1/q}
	\end{align*}
	by the Sobolev-Poincar\'e inequality \eqref{eq:sobolev poincare inequality}.
	Inserting this into \eqref{eq:relative capacity, using triangle inequality}, we get
	\[
	\int_{2B}|u-a|\,d\mu\le 2(C_P+C_{SP}) r\Vert Du\Vert (4\lambda B).
	\]
	Inserting this into \eqref{eq:capacity estimate}, we then get
	\[
	\capa_{\BV}(B\cap S)\le 3(C_P+C_{SP}) \frac{r+1}{a}\Vert Du\Vert(4\lambda B).
	\]
	Recalling the definition of $a$, this implies
	\[
	\left(\,\vint{2B}u^q\,d\mu\right)^{1/q}
	\le  3(C_P+C_{SP})\frac{r+1}{\capa_{\BV}(B\cap S)}\Vert Du\Vert(4\lambda B)
	\]
	provided that $\capa_{\BV}(B\cap S)>0$.
	Next we drop the nonnegativity assumption of $u\in\BV(X)$.
	We have $u=u_+-u_-$ with $u_+,u_-\in\BV(X)$.
	Letting $S_+:=\{u_+^{\wedge}=u_+^{\vee}=0\}$ and $S_-:=\{u_-^{\wedge}=u_-^{\vee}=0\}$,
	we clearly have $S\subset S_+$ and $S\subset S_-$, and so
	\begin{align*}
	&\left(\,\vint{2B}|u|^q\,d\mu\right)^{1/q}
	\le \left(\,\vint{2B}u_+^q\,d\mu\right)^{1/q}+\left(\,\vint{2B}u_-^q\,d\mu\right)^{1/q}\\
	&\le \frac{3(C_P+C_{SP})(r+1)}{\capa_{\BV}(B\cap S_+)}\Vert Du_+\Vert(4\lambda B)+
	\frac{3(C_P+C_{SP})(r+1)}{\capa_{\BV}(B\cap S_-)}\Vert Du_-\Vert(4\lambda B)\\
	&\le \frac{3(C_P+C_{SP})(r+1)}{\capa_{\BV}(B\cap S)}\Vert Du\Vert(4\lambda B)
	\end{align*}
	provided that $\capa_{\BV}(B\cap S)>0$,
	as by the coarea formula \eqref{eq:coarea} it is easy to check that
	$\Vert Du\Vert(4\lambda B)=\Vert Du_+\Vert(4\lambda B)+\Vert Du_-\Vert(4\lambda B)$.
	This completes the proof of the first inequality of the theorem.
	The second is proved similarly; we just need to drop
	the term $\int_X |v|\, d\mu$
	from \eqref{eq:capacity estimate} and proceed as above.
\end{proof}

Using the above Maz'ya-type inequalities, we can now show the following
Poincar\'e inequality for $\BV$ functions with zero boundary values.
The proof is again similar to
the one for Newton-Sobolev functions, see \cite[Corollary 5.54]{BB}.

\begin{corollary}\label{cor:poincare inequality for BV0}
	Let $D\subset X$ be a bounded set with $\capa_1(X\setminus D)>0$. Then
	there is $C_D=C_D(C_d,C_P,\lambda,D)>0$ such that for all $u\in \BV_0(D)$,
	\[
	\int_X |u|\,d\mu\le C_D \Vert Du\Vert(D).
	\]
\end{corollary}
If $D$ is $\mu$-measurable, the integral on the left-hand side can be taken with respect to $D$.

\begin{proof}
	Since $D$ is bounded, we can take a ball $B(x,r)\supset D$.
	By \eqref{eq:Newtonian and BV capacities are comparable} we know that
	$\capa_{\BV}(X\setminus D)>0$, and by \eqref{eq:continuity of BVcap} we can conclude that
	$\capa_{\BV}(B(x,r)\setminus D)>0$ by making $r$ larger, if necessary.
	Take $u\in\BV_0(D)$.
	By \eqref{eq:null sets of Hausdorff measure and capacity} and
	\eqref{eq:Newtonian and BV capacities are comparable},
	\begin{equation}\label{eq:capacity of zero set and complement of D}
	\begin{split}
	\capa_{\BV}(B(x,r)\setminus D)
	&=\capa_{\BV}(B(x,r)\cap \{u^{\wedge}=u^{\vee}=0\}\setminus D)\\
	&\le \capa_{\BV}(B(x,r)\cap \{u^{\wedge}=u^{\vee}=0\}).
	\end{split}
	\end{equation}
	By H\"older's inequality and the Maz'ya-type inequality \eqref{eq:first mazya inequality},
	\begin{align*}
	\frac{1}{\mu(B(x,2r))}\int_X|u|\,d\mu
	&= \vint{B(x,2r)}|u|\,d\mu\\
	&\le \left(\,\vint{B(x,2r)}|u|^{Q/(Q-1)}\,d\mu\right)^{(Q-1)/Q}\\
	&\le \frac{3(C_P+C_{SP})(r+1)}{\capa_{\BV}(B(x,r)\cap \{u^{\wedge}=u^{\vee}=0\})}
	\Vert Du\Vert(B(x,4\lambda r))\\
	&\le \frac{3(C_P+C_{SP})(r+1)}{\capa_{\BV}(B(x,r)\setminus D)}
	\Vert Du\Vert(B(x,4\lambda r))\qquad\textrm{by }\eqref{eq:capacity of zero set and complement of D}\\
	&= \frac{3(C_P+C_{SP})(r+1)}{\capa_{\BV}(B(x,r)\setminus D)}
	\Vert Du\Vert(D)
	\end{align*}
	by Proposition \ref{prop:variation measure concentration for BV0}.
	Thus we can choose
	\[
	C_D=\frac{3(C_P+C_{SP})(r+1)\mu(B(x,2r))}{\capa_{\BV}(B(x,r)\setminus D)}.
	\]
\end{proof}

Now we can prove the following property of the variational $\BV$-capacity.
Combined with Proposition \ref{prop:basic properties of the variational BV-capacity}, this shows that
$\rcapa_{\BV}(\cdot,D)$ is an outer measure on the subsets of $D$.

\begin{proposition}
If $D\subset X$ is bounded and $A_1,A_2,\ldots\subset D$, then
\[
\rcapa_{\BV}\left(\bigcup_{j=1}^{\infty}A_j,D\right)\le \sum_{j=1}^{\infty}\rcapa_{\BV}(A_j,D).
\]
\end{proposition}

\begin{proof}
We can assume that the right-hand side is finite. Fix $\eps>0$.
For each $j\in\N$, choose $u_j\in\BV_0(D)$
such that $0\le u_j\le 1$, $u_j^{\wedge}=1$ on $A_j$, and
\[
\Vert Du_j\Vert(X)\le \rcapa_{\BV}(A_j,D)+2^{-j}\eps.
\]
Consider first the case $\capa_1(X\setminus D)=0$. Let
$u:=\min\left\{1,\sum_{j=1}^{\infty}u_j\right\}$, so that $u^{\wedge}= 1$ on $\bigcup_{j=1}^{\infty}A_j$.
By Lebesgue's dominated convergence theorem,
$\min\left\{1,\sum_{j=1}^{N}u_j\right\}\to u$ in $L^1(X)$
as $N\to\infty$. Thus by lower semicontinuity of the total variation with respect to
$L^1$-convergence, we get
\begin{align*}
\Vert Du\Vert(X)
&\le \liminf_{N\to\infty}\Vert D\left(\min\left\{1,\sum_{j=1}^{N}u_j\right\}\right)\Vert(X)\\
&\le \liminf_{N\to\infty}\Vert D\left(\sum_{j=1}^{N}u_j\right)\Vert(X)\\
&\le \sum_{j=1}^{\infty}\Vert Du_j\Vert(X)\quad\textrm{by }\eqref{eq:BV functions form vector space}\\
&\le \sum_{j=1}^{\infty}\rcapa_{\BV}(A_j,D)+\eps.
\end{align*}
Thus $u\in\BV(X)$ and then obviously $u\in\BV_0(D)$, since $\capa_1(X\setminus D)=0$.
Thus we get
\[
\rcapa_{\BV}\left(\bigcup_{j=1}^{\infty}A_j,D\right)\le \Vert Du\Vert(X)
\le \sum_{j=1}^{\infty}\rcapa_{\BV}(A_j,D)+\eps.
\]
Letting $\eps\to 0$, we obtain the result.

Then consider the case $\capa_1(X\setminus D)>0$.
By Corollary \ref{cor:poincare inequality for BV0}, there is a constant
$C_D>0$ such that $\Vert u_j\Vert_{L^1(X)}\le C_D \Vert Du_j\Vert(X)$ for all $j\in\N$.
Thus $\sum_{j=1}^{\infty}\Vert u_j\Vert_{\BV(X)}<\infty$, so that defining
$u:=\sum_{j=1}^{\infty}u_j$, by Proposition \ref{prop:BV0 closed subspace} we have that
$u\in\BV_0(D)$. Clearly $u^{\wedge}\ge 1$ on $\bigcup_{j=1}^{\infty}A_j$.
Thus
\[
\rcapa_{\BV}\left(\bigcup_{j=1}^{\infty}A_j,D\right)\le
\Vert Du\Vert(X)\le \sum_{j=1}^{\infty}\Vert Du_j\Vert(X)
\le \sum_{j=1}^{\infty}\rcapa_{\BV}(A_j,D)+\eps.
\]
Again letting $\eps\to 0$, we obtain the result.
\end{proof}

The result given in the following theorem is perhaps unexpected, since the class of
admissible test functions for
$\rcapa_{\BV}$ is so much larger than the class of admissible test functions for
$\rcapa_{\mathrm{lip}}$.
Previously, a similar result was given in \cite[Theorem 4.3]{HaSh}, but there the
variational $\BV$-capacity
$\rcapa_{\BV}(K,\Om)$ was defined by requiring
the test functions to be compactly supported in $\Om$ and to
take the value $1$ in a \emph{neighborhood} of $K$.
We need to obtain these two properties by using our previous results, but after that
we employ similar methods as in \cite{HaSh}.

\begin{theorem}\label{thm:BV and lip caps are equal}
Let $\Om\subset X$ be open and let $K\subset \Om$  be compact. Then
\[
\rcapa_{\mathrm{lip}}(K,\Om)=\rcapa_{\BV}(K,\Om).
\]
\end{theorem}

\begin{proof}
One inequality is clear. To prove the opposite inequality, we can assume that $\rcapa_{\BV}(K,\Om)<\infty$.
Fix $\eps>0$.
By Theorem \ref{thm:BV capacity is outer}, we find
an open set $V$ such that $K\subset V\subset \Om$ and
$\rcapa_{\BV}(V,\Om)< \rcapa_{\BV}(K,\Om)+\eps$.
Then we find $u\in \BV_0(\Om)$ such that $0\le u\le 1$,
$u^{\wedge}=1$ on $V$, and $\Vert D u\Vert(X)<\rcapa_{\BV}(K,\Om)+\eps$.
By Proposition \ref{prop:weak density of lipschitz functions}, we find functions
$u_i\in\Lip_c(\Om)$ such that
\begin{equation}\label{eq:properties of wi}
\Vert u_i-u\Vert_{L^1(X)}<1/i\quad \textrm{and}\quad \int_X g_{u_i}\,d\mu<\Vert Du\Vert(X)+1/i,
\quad i\in\N.
\end{equation}
Take $\eta\in \Lip_c(V)$ with $0\le \eta\le 1$ on $X$ and $\eta=1$ on $K$.
Then let
\[
v_i:=\eta+(1-\eta)u_i,\quad i\in\N.
\]
We have $v_i\in\Lip_c(\Om)$ and $v_i=1$ on $K$,
so that each $v_i$ is admissible for $\rcapa_{\mathrm{lip}}(K,\Om)$.
By a Leibniz rule, see \cite[Lemma 2.18]{BB}, we have
\[
g_{v_i}\le (1-\eta)g_{u_i}+|1-u_i|g_{\eta},
\]
and so by \eqref{eq:properties of wi},
\begin{align*}
\int_X g_{v_i}\,d\mu
&\le \int_X g_{u_i}\,d\mu+\sup_X g_{\eta}\int_V |1-u_i|\,d\mu\\
&\le \Vert Du\Vert(X)+1/i+\sup_X g_{\eta}\int_V |1-u_i|\,d\mu\\
&= \Vert Du\Vert(X)+1/i+\sup_X g_{\eta}\int_V |u-u_i|\,d\mu\\
&\to \Vert Du\Vert(X)\quad\textrm{as }i\to\infty.
\end{align*}
Thus for some sufficiently large index $i\in\N$, we have
\[
\rcapa_{\mathrm{lip}}(K,\Om)\le \int_X g_{v_i}\,d\mu
\le \Vert Du\Vert(X)+\eps
<\rcapa_{\BV}(K,\Om)+2\eps.
\]
Letting $\eps\to 0$, we conclude the proof.
\end{proof}

\begin{example}\label{ex:comparison of capacities}
Let $X=\R$ (unweighted) and choose
\[
K=[-2,-1]\cup [1,2]
\quad\textrm{and}\quad
\Om=(-3,0)\cup(0,3).
\]
Then it is straightforward to show that
\[
\rcapa_{\mathrm{lip}}(K,\Om)=\rcapa_1(K,\Om)=\rcapa_{\BV}(K,\Om)=4.
\] 
On the other hand, if we defined $\rcapa_{\BV}(K,\Om)$ by only requiring the test
functions to satisfy $u=0$
on $\R\setminus \Om$ (almost everywhere), then
we would have $\rcapa_{\BV}(K,\Om)=2$.
In this sense, our current definition of $\rcapa_{\BV}$ can be considered to be the natural one.

Moreover, if we defined $\rcapa_{\BV}(K,\Om)$ by requiring the test functions to satisfy
$u^{\vee}\ge 1$ on $K$ instead of $u^{\wedge}\ge 1$, we would get a generally smaller
but comparable quantity, see \cite[Example 4.4, Theorem 4.5, Theorem 4.6]{HaSh}.
\end{example}

More generally, consider $\rcapa_{\mathrm{lip}}(A,D)$ for $A\subset D\subset X$.
Recall from Example \ref{ex:first comparisons of capacities} that it is possible to have
$\rcapa_1(A,D)=0$ but $\rcapa_{\mathrm{lip}}(A,D)=\infty$.
According to \cite[Example 6.1]{BB-cap}, when $1<p<\infty$ it is also possible to have
\[
0<\rcapa_p(K,D)<\rcapa_{\mathrm{lip},p}(K,D)<\infty
\]
for a compact set $K\subset\inte D$, where $\rcapa_{\mathrm{lip},p}$ is defined
by requiring the test functions in the definition of $\rcapa_p$  to be Lipschitz.

\begin{openproblem}
Do we have
$\rcapa_1(K,D)=\rcapa_{\mathrm{lip}}(K,D)$
for every $D\subset X$ and compact $K\subset \inte D$?
\end{openproblem}

The following example shows that $\rcapa_{\BV}(K,D)$ and $\rcapa_1(K,D)$ can differ
even for a compact $K\subset \inte D$.

\begin{example}\label{ex:weak density fails}
Let $X=\R$, define a weight function $w:=1+\ch_{[0,3]}$,
and let $d\mu:=w\,d\mathcal L^1$.
Choose $D=[0,3]$, $K=[1,2]$, and $u:=\ch_{[0,3]}$, so that $u\in\BV_0(D)$. It is easy to check
that $\Vert Du\Vert(X)=2$, and so $\rcapa_{\BV}(K,D)\le 2$. On the other hand, clearly
\[
\liminf_{i\to\infty}\int_{\R}g_{u_i}\,d\mu=2\liminf_{i\to\infty}\int_{\R}g_{u_i}\,d\mathcal L^1
\ge 4
\]
for every sequence of functions $(u_i)\subset\liploc(X)$ with $u_i\to u$ in $L^1(X)$
and $\supp u_i \subset D$.
Thus Proposition \ref{prop:weak density of lipschitz functions} can fail if $\Om$ is not open.

Similarly, $\int_{\R}g_v\,d\mu\ge 4$ for every $v\in N^{1,1}_0(D)$ with $v=1$ on $K$. Thus
$\rcapa_1(K,D)\ge 4$ (in fact, equality holds).
Thus $\rcapa_{\BV}(K,D)<\rcapa_1(K,D)$.
\end{example}

\paragraph{Acknowledgments.}The research was
funded by a grant from the Finnish Cultural Foundation.
Most of the research for this paper was conducted during the author's visit to the University of Cincinnati,
whose hospitality the author wishes to acknowledge.
The author also wishes to thank Nageswari Shanmugalingam for helping to derive Maz'ya-type inequalities
for $\BV$ functions, and
Anders and Jana Bj\"orn for discussions on variational $p$-capacities.

\noindent Address:\\

\noindent Department of Mathematics\\
Link\"oping University\\
SE-581 83 Link\"oping, Sweden\\
E-mail: {\tt panu.lahti@aalto.fi}


\begin{thebibliography}{ACMM}

\bibitem{A1}L. Ambrosio,
\textit{Fine properties of sets of finite perimeter in doubling metric measure spaces},
Calculus of variations, nonsmooth analysis and related topics.
Set-Valued Anal. 10 (2002), no. 2-3, 111--128.

\bibitem{AFP}L. Ambrosio, N. Fusco, and D. Pallara,
\textit{Functions of bounded variation and free discontinuity problems.}
Oxford Mathematical Monographs. The Clarendon Press, Oxford University Press, New York, 2000.

\bibitem{AMP}L. Ambrosio, M. Miranda, Jr., and D. Pallara,
\textit{Special functions of bounded variation in doubling metric measure spaces},
Calculus of variations: topics from the mathematical heritage of E. De Giorgi, 1--45,
Quad. Mat., 14, Dept. Math., Seconda Univ. Napoli, Caserta, 2004.

\bibitem{BS}L. Beck and T. Schmidt,
\textit{Convex duality and uniqueness for BV-minimizers},
J. Funct. Anal. 268 (2015), no. 10, 3061--3107.

\bibitem{BB}A. Bj\"orn and J. Bj\"orn,
\textit{Nonlinear potential theory on metric spaces},
EMS Tracts in Mathematics, 17. European Mathematical Society (EMS), Z\"urich, 2011. xii+403 pp.

\bibitem{BB-OD}A. Bj\"orn and J. Bj\"orn,
\textit{Obstacle and Dirichlet problems on arbitrary nonopen sets in metric spaces, and fine topology},
Rev. Mat. Iberoam. 31 (2015), no. 1, 161--214.

\bibitem{BB-cap}A. Bj\"orn and J. Bj\"orn,
\textit{The variational capacity with respect to nonopen sets in metric spaces},
Potential Anal. 40 (2014), no. 1, 57--80.

\bibitem{BBL-SS}A. Bj\"orn, J. Bj\"orn, and V. Latvala,
\textit{Sobolev spaces, fine gradients and quasicontinuity on quasiopen sets},
Ann. Acad. Sci. Fenn. Math. 41 (2016), no. 2, 551--560.

\bibitem{BBS}A. Bj\"orn, J. Bj\"orn, and N. Shanmugalingam,
\textit{Quasicontinuity of Newton-Sobolev functions and density of Lipschitz functions on metric spaces}, 
Houston J. Math. 34 (2008), no. 4, 1197--1211. 

\bibitem{BBS2}A. Bj\"orn, J. Bj\"orn, and N. Shanmugalingam,
\textit{The Dirichlet problem for p-harmonic functions on metric spaces},
J. Reine Angew. Math. 556 (2003), 173--203.

\bibitem{BBS3}A. Bj\"orn, J. Bj\"orn, and N. Shanmugalingam,
\textit{The Dirichlet problem for p-harmonic functions with respect to the Mazurkiewicz boundary, and new capacities},
J. Differential Equations 259 (2015), no. 7, 3078--3114.

\bibitem{BDG}E. Bombieri, E. De Giorgi, and E. Giusti,
\textit{Minimal cones and the Bernstein problem},
Invent. Math. 7 1969 243--268. 

\bibitem{CDLP}M. Carriero, G. Dal Maso, A. Leaci, and E. Pascali,
\textit{Relaxation of the nonparametric plateau problem with an obstacle},
J. Math. Pures Appl. (9) 67 (1988), no. 4, 359--396. 

\bibitem{EvaG92}L. C. Evans and R. F. Gariepy,
\textit{Measure theory and fine properties of functions},
Studies in Advanced Mathematics series, CRC Press, Boca Raton, 1992.

\bibitem{Fed}H. Federer,
\textit{Geometric measure theory},
Die Grundlehren der mathematischen Wissenschaften, Band 153 Springer-Verlag New York Inc., New York 1969 xiv+676 pp.

\bibitem{FedZie}H. Federer and W. P. Ziemer,
\textit{The Lebesgue set of a function whose distribution derivatives are p-th power summable},
Indiana Univ. Math. J. 22 (1972/73), 139--158.

\bibitem{Giu84}E. Giusti,
\textit{Minimal surfaces and functions of bounded variation},
Monographs in Mathematics, 80. Birkh\"auser Verlag, Basel, 1984. xii+240 pp.

\bibitem{HaKi}H. Hakkarainen and J. Kinnunen,
\textit{The BV-capacity in metric spaces},
Manuscripta Math. 132 (2010), no. 1-2, 51--73.

\bibitem{HKL}H. Hakkarainen, J. Kinnunen, and P. Lahti,
\textit{Regularity of minimizers of the area functional in metric spaces},
Adv. Calc. Var. 8 (2015), no. 1, 55--68.

\bibitem{HKLL}H. Hakkarainen, J. Kinnunen, P. Lahti, and P. Lehtel\"a,
\textit{Relaxation and integral representation for functionals of linear
growth on metric measures spaces},
Anal. Geom. Metr. Spaces 4 (2016), 288--313. 

\bibitem{HKLS}H. Hakkarainen, R. Korte, P. Lahti, and N. Shanmugalingam,
\textit{Stability and continuity of functions of least gradient},
Anal. Geom. Metr. Spaces 3 (2015), 123--139. 

\bibitem{HaSh}H. Hakkarainen and N. Shanmugalingam,
\textit{Comparisons of relative BV-capacities and Sobolev capacity in metric spaces},
Nonlinear Anal. 74 (2011), no. 16, 5525--5543. 

\bibitem{Hei}J. Heinonen,
\textit{Lectures on analysis on metric spaces},
Universitext. Springer-Verlag, New York, 2001. x+140 pp.

\bibitem{HKM}J. Heinonen, T. Kilpel\"ainen, and O. Martio,
\textit{Nonlinear potential theory of degenerate elliptic equations},
Unabridged republication of the 1993 original. Dover Publications, Inc., Mineola, NY, 2006. xii+404 pp.

\bibitem{HK}J. Heinonen and P. Koskela,
\textit{Quasiconformal maps in metric spaces with controlled geometry},
Acta Math. 181 (1998), no. 1, 1--61.

\bibitem{KKST2}J. Kinnunen, R. Korte, N. Shanmugalingam, and H. Tuominen,
\textit{Lebesgue points and capacities via the boxing inequality in metric spaces},
Indiana Univ. Math. J. 57 (2008), no. 1, 401--430.

\bibitem{KKST3}J. Kinnunen, R. Korte, N. Shanmugalingam, and H. Tuominen,
\textit{Pointwise properties of functions of bounded variation in metric spaces},
Rev. Mat. Complut. 27 (2014), no. 1, 41--67.

\bibitem{KKST-DG}J. Kinnunen, R. Korte, N. Shanmugalingam, and H. Tuominen,
\textit{The De Giorgi measure and an obstacle problem related to minimal surfaces in metric spaces},
J. Math. Pures Appl. (9) 93 (2010), no. 6, 599--622. 

\bibitem{KLLS}R. Korte, P. Lahti, X. Li, and N. Shanmugalingam,
\textit{Notions of Dirichlet problem for functions of least gradient in metric measure spaces},
preprint 2016. https://arxiv.org/abs/1612.06078

\bibitem{L3}P. Lahti,
\textit{A Federer-style characterization of sets of finite perimeter on metric spaces},
preprint 2016. https://arxiv.org/abs/1612.06286

\bibitem{L-FC}P. Lahti,
\textit{A notion of fine continuity for BV functions on metric spaces},
Potential Anal. 46 (2017), no. 2, 279--294. 

\bibitem{L-WC}P. Lahti,
\textit{Superminimizers and a weak Cartan property for $p=1$ in metric spaces},
preprint 2017. https://arxiv.org/abs/1706.01873

\bibitem{LaSh}P. Lahti and N. Shanmugalingam,
\textit{Fine properties and a notion of quasicontinuity for $\BV$ functions on metric spaces},
J. Math. Pures Appl. (9) 107 (2017), no. 2, 150--182. 

\bibitem{LaSh2}P. Lahti and N. Shanmugalingam,
\textit{Trace theorems for functions of bounded variation in metric spaces},
preprint 2015. https://arxiv.org/abs/1507.07006

\bibitem{MZ}J. Mal\'{y} and W. Ziemer,
\textit{Fine regularity of solutions of elliptic partial differential equations},
Mathematical Surveys and Monographs, 51. American Mathematical Society, Providence, RI, 1997. xiv+291 pp.

\bibitem{MRL}J. M. Maz\'on, J. D. Rossi, and S. S. De Le\'on,
\textit{Functions of least gradient and $1$-harmonic functions},
Indiana Univ. Math. J. 63 No. 4 (2014), 1067--1084.

\bibitem{Maz2}V. G. Maz'ya,
\textit{Sobolev spaces},
Translated from the Russian by T. O. Shaposhnikova. Springer Series in Soviet Mathematics.
Springer-Verlag, Berlin, 1985. xix+486 pp.

\bibitem{Maz1}V. G. Maz'ya,
\textit{The Dirichlet problem for elliptic equations of arbitrary order in unbounded domains}, 
Dokl. Akad. Nauk SSSR 150 1963 1221--1224.

\bibitem{M}M.~Miranda, Jr.,
\textit{Functions of bounded variation on ``good'' metric spaces},
J. Math. Pures Appl. (9) 82  (2003),  no. 8, 975--1004.

\bibitem{Mak}T. M\"ak\"al\"ainen,
\textit{Adams inequality on metric measure spaces},
Rev. Mat. Iberoam. 25 (2009), no. 2, 533--558.

\bibitem{S-H}N. Shanmugalingam,
\textit{Harmonic functions on metric spaces},
Illinois J. Math. 45 (2001), no. 3, 1021--1050. 

\bibitem{S}N. Shanmugalingam,
\textit{Newtonian spaces: An extension of {S}obolev spaces to metric measure spaces},
Rev. Mat. Iberoamericana 16(2) (2000), 243--279.

\bibitem{SWZ}P. Sternberg, G. Williams, and W. Ziemer,
\textit{Existence, uniqueness, and regularity for functions of least gradient},
J. Reine Angew. Math. 430 (1992), 35--60. 

\bibitem{Zie89}W. Ziemer,
\textit{Weakly differentiable functions. Sobolev spaces and functions of bounded variation},
Graduate Texts in Mathematics, 120. Springer-Verlag, New York, 1989.

\bibitem{ZZ}W. Ziemer and K. Zumbrun,
\textit{The obstacle problem for functions of least gradient},
Math. Bohem. 124 (1999), no. 2-3, 193--219. 

\end{thebibliography}
\end{document}